\documentclass[twoside,a4paper,12pt,centertags]{amsart}
\usepackage{amsmath,amssymb,verbatim,vmargin}
\usepackage[colorlinks,citecolor=red,pagebackref,hypertexnames=false]{hyperref}

\theoremstyle{plain}
\newtheorem{thm}{Theorem}[section]
\newtheorem{lem}[thm]{Lemma}
\newtheorem{cor}[thm]{Corollary}
\newtheorem{prop}[thm]{Proposition}

\theoremstyle{definition}
\newtheorem{defn}[thm]{Definition}
\newtheorem{rem}[thm]{Remark}

\mathchardef\semic="303B
\newcommand{\dirac}{{\mathbf D}}

\newcommand{\R}{{\mathbb R}}
\newcommand{\N}{{\mathbb N}}
\newcommand{\C}{{\mathbb C}}

\newcommand{\mH}{{\mathcal H}}

\newcommand{\mL}{{\mathcal L}}

\newcommand{\mA}{{\mathcal A}}
\newcommand{\mE}{{\mathcal E}}
\newcommand{\mD}{{\mathcal D}}

\newcommand{\mR}{{\mathcal R}}
\newcommand{\mT}{{\mathcal T}}
\newcommand{\mS}{{\mathcal S}}

\DeclareMathOperator{\re}{Re}

\newcommand{\nul}{\textsf{N}}
\newcommand{\ran}{\textsf{R}}
\newcommand{\dom}{\textsf{D}}

\newcommand{\clos}[1]{\overline{#1}}
\newcommand{\conj}[1]{\overline{#1}}

\newcommand{\sgn}{\text{{\rm sgn}}}
\newcommand{\barint}{\mbox{$ave \int$}}
\newcommand{\divv}{{\text{{\rm div}}}}
\newcommand{\curl}{{\text{{\rm curl}}}}
\newcommand{\esssup}{\text{{\rm ess sup}}}

\newcommand{\tdd}[2]{\tfrac{\partial #1}{\partial #2}}

\newcommand{\ta}{{\scriptscriptstyle \parallel}}
\newcommand{\no}{{\scriptscriptstyle\perp}}
\newcommand{\pd}{\partial}

\newcommand{\uT}{{\underline T}}

\newcommand{\loc}{\text{{\rm loc}}}
\newcommand{\tN}{\widetilde N_*}

\newcommand{\bx}{{\bf x}}
\newcommand{\by}{{\bf y}}
\newcommand{\bz}{{\bf z}}

\newcommand{\rn}{\mathbb{R}^n}

\newcommand{\reu}{\mathbb{R}^{1+n}_+}
\newcommand{\ree}{\mathbb{R}^{1+n}}

\def\barint_#1{\mathchoice
            {\mathop{\vrule width 6pt
height 3 pt depth -2.5pt
                    \kern -8.8pt
\intop \kern -4pt}\nolimits_{#1}}%
            {\mathop{\vrule width 5pt height
3 pt depth -2.6pt
                    \kern -6.5pt
\intop \kern -4pt}\nolimits_{#1}}%
            {\mathop{\vrule width 5pt height
3 pt depth -2.6pt
                    \kern -6pt
\intop \kern -4pt}\nolimits_{#1}}%
            {\mathop{\vrule width 5pt height
3 pt depth -2.6pt
          \kern -6pt \intop \kern -4pt}\nolimits_{#1}}}
          
          \def\bariint_#1{\mathchoice
            {\mathop{\vrule width 10pt
height 3 pt depth -2.5pt
                    \kern -12.8pt
\intop \kern -10pt\intop \kern -4pt}\nolimits_{#1}}%
            {\mathop{\vrule width 9pt height
3 pt depth -2.6pt
                    \kern -10.5pt
\intop \kern -10pt\intop \kern -4pt}\nolimits_{#1}}%
            {\mathop{\vrule width 9pt height
3 pt depth -2.6pt
                    \kern -10pt
\intop \kern -10pt\intop \kern -4pt}\nolimits_{#1}}%
            {\mathop{\vrule width 9pt height
3 pt depth -2.6pt
          \kern -10pt \intop \kern -10pt\intop \kern -4pt}
      \nolimits_{  #1}}}

\usepackage{color}

\definecolor{gr}{rgb}   {0.,   0.8,   0. } 
\definecolor{bl}{rgb}   {0.,   0.5,   1. } 
\definecolor{mg}{rgb}   {0.7,  0.,    0.7}

\makeatletter
\@namedef{subjclassname@2010}{
  \textup{2010} Mathematics Subject Classification}
\makeatother

\title[Rellich estimates and solvability]{Boundary layers, Rellich estimates and extrapolation of solvability for elliptic systems}
\author{Pascal Auscher}
\address{Univ. Paris-Sud, laboratoire de Math\'ematiques, UMR 8628, F-91405 {\sc Orsay}; CNRS, F-91405 {\sc Orsay}} 
\email{pascal.auscher@math.u-psud.fr}
\author{Mihalis Mourgoglou}
\address{Univ. Paris-Sud, laboratoire de Math\'ematiques, UMR 8628, F-91405 {\sc Orsay}; CNRS, F-91405 {\sc Orsay}} 
\email{mihalis.mourgoglou@math.u-psud.fr}

\begin{document}

\begin{abstract}
The purpose of this article is to study extrapolation of  solvability for boundary value problems 
of elliptic systems in divergence form on the upper half-space assuming De Giorgi type conditions. We develop a method allowing to treat each boundary value problem independently of the others.  
 We shall base our study on solvability for energy solutions, estimates for  boundary layers, equivalence of certain boundary estimates with interior control so that solvability reduces to a one-sided Rellich inequality. Our  method then amounts to extrapolating this Rellich inequality using atomic Hardy spaces, interpolation and duality. In the way, we {reprove the Regularity-Dirichlet  duality principle between dual systems and extend it} to  $H^1-BMO$. We also exhibit and use a similar  Neumann-Neumann duality  principle.
\end{abstract}



\maketitle

\tableofcontents

\section{Introduction}

Boundary value problems  for second order elliptic equations have a long history. The breakthroughs of Dahlberg \cite{Da} for  the Laplace equation on Lipschitz domains  {and the boundedness of the corresponding layer potentials by Coifman, McIntosh and Meyer \cite{CMcM}}  opened the door  to a thorough study of such problems, generalizing domains or operators. By flattening  the boundary, one instead looks at equations with measurable coefficients and  considers two types of domains, either the upper-half space as a prototype for unbounded domains or the unit ball as a prototype for a bounded domain. There one can study boundary value problems with different types of data spaces. All of this is well explained in the book by Kenig \cite{Ke}. Solving these boundary value problems can be a difficult task; there is no comprehensive nor unified treatment  of this issue at this time.  Let us just mention that the solution of the Kato conjecture \cite{AHLMcT} and its  developments gave rise to new estimates and  new methods so that progress in the area is rather impressive as of now.

The purpose of this article is to study extrapolation of boundary value problems  
for elliptic systems in divergence form on the upper half-space $\reu$, $1+n\ge 2$. Extrapolation means that, 
assuming the problem can be solved  for  some space $X$ of data, one can push the solvability range to some other spaces.  
For Regularity and Neumann problems, $X$ is an $L^r$ space, $r>1$ and  one extrapolates to $L^p$ for  $1<p<r$ and  $H^p$ data for some range of $p$ below 1. For Dirichlet problems, one starts from $L^q$ for some $q<\infty$ and extrapolates to $L^p$ for $q<p<\infty$, BMO and H\"older spaces up to some exponent. In fact, one can see the Dirichlet problem as a Regularity problem in   spaces of data with regularity exponent -1. One can  also formulate Neumann problems in spaces of data with regularity -1.  

We do not treat here the openness property of extrapolation, that is that solvability at one space of data can be perturbed to nearby spaces in the given scale. This will  be  treated in \cite{AS} using further developments.

These types of extrapolation results are not new, at least for equations, starting from the seminal works of \cite{DaK} for the Laplace equation on Lipschitz domains and \cite{KP} for real symmetric equations. Further  contributions are in \cite{Br} for the Laplace equation looking at $H^p$ data for $p<1$,  in \cite{Di} in the context of the Laplace equation on smooth domains of Riemannian manifolds and in  \cite{DK} for real equations on bounded Lipschitz  domains.  See also some comments in \cite{HKMP2} outlining   
a strategy {using Kalton-Mitrea extrapolation \cite{KM}} when  layer potentials associated to the 
operators are invertible. Of course, we are just mentioning the works related to extrapolation in this subject and not the numerous ones on solvability for second order elliptic operators under various assumptions. In some sense, we are after a sort of extrapolation reminiscent to the  Calder\'on and Zygmund extrapolation for singular integrals because the operators under considerations can be thought of as generalized singular integrals. 

To do so, we introduce a new method which allows to treat each boundary value problem independently of the other ones and  to consider systems and not just equations,
{assuming De Giorgi-Nash type local H\"older regularity, in the  
interior and for reflections across the boundary}.  For  the Neumann problem, this is completely new: {in \cite{KP}, which is the closest
antecedent to our results here, extrapolation of solvability for the Neumann problem 
was linked to that of the Regularity problem}. Our method will clarify the  Regularity-Dirichlet duality principle for solvability {obtained in \cite{HKMP2}},  extend its  range  to $H^1-BMO$ and we will also   formulate {and use} a new duality principle for Neumann problems.  
Also our exponents are explicitly determined by the ones in the De Giorgi conditions.  
{Our strategy can be summarized as follows: try to distinguish as much as possible interior and boundary estimates so  as to use  \textit{a priori} estimates most of the time. To do so, we  have to reverse the order in which we use some tools compared to other works.}

Our divergence operators will be precisely defined in Section \ref{sec:general} and ellipticity
 will be taken in the sense of some G\aa rding inequality. The boundary value problems are treated for operators whose coefficients do not depend on $t$, the transverse  variable to the boundary,  but some results do not need this. For the purpose of the introduction, it is best to assume $t$-independence. 

We shall rely on energy solutions. Indeed,  Regularity and Neumann problems are always well-posed (modulo constants)  in the energy class without any further information. We deviate here from the treatment done in \cite{KR} or \cite{HKMP2} by using the ``natural'' energy space given by the Dirichlet integral $\int |\nabla u|^2$. Even in the unbounded situation of the upper half-space, things turn out to work rather well with this space.  Solvability  of Regularity and Neumann problems   means here that energy solutions satisfy the required inequality respectively: control of $\|\tN(\nabla u)\|_{p}$, the $L^p$ norm of (a modified)  non-tangential maximal function of $\nabla u$, by  $\|\nabla_{x}u|_{t=0}\|_{p}$, where  $\nabla_{x}u|_{t=0}$ is 
the tangential gradient at the boundary  or by  $\|\pd_{\nu_{A}}u|_{t=0}\|_{p}$, where $\pd_{\nu_{A}}u|_{t=0}$ is  the conormal derivative at the boundary. For the boundary estimates, the norm $\|\ \|_{p}$ denotes an $L^p$ norm if $p>1$ and a Hardy $H^p$ norm if $p\le 1$.

 One of the main results here  is the following.  Assuming {interior} De Giorgi type conditions, {there is} 
 an \textit{a priori} equivalence between  $\|\tN(\nabla u)\|_{p}$ and the sum $\|\pd_{\nu_{A}}u|_{t=0}\|_{p}+  \|\nabla_{x}u|_{t=0}\|_{p}$ in a range $1-\varepsilon<p\le 2$ for  energy solutions (and for other types of solutions as well) with $\varepsilon$ specified by our assumptions.\footnote{We shall not attempt  to treat here the range $2<p<2+\varepsilon'$.  In fact,  it can be shown to hold without the De Giorgi condition, at the expense of some further work which shall be presented in \cite{AS}.}   For $p=2$, this was one of the key result in \cite{AA}. The bound from below holds for any weak solution: it was known in the range $1<p<\infty$ from \cite{KP}  and has been proved recently in a range $1-\varepsilon<p \le 1$ in \cite{HMiMo}. The bound from above has been addressed in the range $1<p<2+\varepsilon'$ in \cite{HKMP2} for a class of solutions  $u$ which can be represented by the layer potentials built in \cite{AAAHK} using Green's  representation formula
\begin{equation}
\label{green}
u(t,x)=S_{t}(\pd_{\nu_{A}}u|_{t=0})(x) - D_{t}(u|_{t=0})(x)
\end{equation}
where $S_{t}$ and $D_{t}$ are respectively the single and double layer potentials associated to $\divv A \nabla$ on $\ree$.
We shall prove (Theorem \ref{thm:main}) that the bound from above holds in the range $1-\varepsilon<p\le 2$ for any solution in the energy class and other classes. Here, we use the newly discovered relation by A. Ros\'en \cite{R1} between the layer potentials and the first order formalism of \cite{AAMc}, which  gives $L^2$ boundedness of the double layer potential and {$L^{2} \to \dot W^{1,2}$ boundedness of the single layer potential in full generality.   It allows to use instead the differentiated form of \eqref{green} which actually comes before  in the analysis (and one does not care about the constant of integration at this stage)
\begin{equation}
\label{green'}
\nabla u(t,x)=\nabla S_{t}(\pd_{\nu_{A}}u|_{t=0})(x) - \nabla D_{t}(u|_{t=0})(x)
\end{equation}
   { and one only needs to  have  two \textit{a priori} bounds} for the layer potentials. The first one is $\|\tN(\nabla D_{t}h)\|_{p} \lesssim \|\nabla_{x}h\|_{p}$ that we obtain in the range $1-\varepsilon<p<2$ (again, recall that boundary norms are $H^p$ norms when $p\le 1$). We note this was proved for $1<p<2+\varepsilon'$ for {complex} equations and $1+n\ge 3$ in \cite[Proposition 5.9]{HKMP2} {and there is an interesting comment to make as an illustration of reversing the order in which we use tools. The argument there uses an estimate for what is called {$L$-harmonic} conjugates (such an estimate is a one-sided Rellich inequality in disguise (see below)) and  seems therefore to be limited to $p>1$. 
Instead, our argument  for this inequality does not require such an estimate; it only uses atomic theory and interpolation.  This is made possible because we use \eqref{green'} and not \eqref{green}{; the Rellich inequality is used later}.
  The other needed \textit{a priori} bound}}
 $\|\tN(\nabla S_{t}h)\|_{p} \lesssim \|h\|_{p}$  in the range $1-\varepsilon<p<2+\varepsilon'$ was known from \cite{HMiMo} at least in the equation case and $1+n\ge 3$.  With this in hand, solvability for a given $p$ in this range is  equivalent to a boundary estimate   of Rellich type (again for energy solutions) which is 
\begin{equation}
\label{R}
\|\pd_{\nu_{A}}u|_{t=0}\|_{p} \lesssim \|\nabla_{x}u|_{t=0}\|_{p}
\end{equation}
for the Regularity problem  and 
\begin{equation}
\label{N}
\|\nabla_{x}u|_{t=0}\|_{p} \lesssim \|\pd_{\nu_{A}}u|_{t=0}\|_{p}
\end{equation}   for the Neumann problem.

The outcome of this is that in order to extrapolate solvability, it suffices to extrapolate a one-sided Rellich inequality. This step, therefore, completely happens at the boundary. Basically, in the spirit of the ideas in \cite{DaK} and \cite{KP}, we can use Hardy space atomic theory on the boundary and interpolation. But the difference is that we only  have to prove  \eqref{R} or \eqref{N} with given data  a 2-atom (Section \ref{sec:rellichp<1})   and we do this by showing that the missing data is a molecule (Section \ref{sec:extra}) {without going back to non-tangential maximal estimates}.
Aside from some pointwise estimates on solutions shown in Section \ref{sec:decay} relying on some form of boundary regularity, this step uses, of course, the initial solvability assumption even to get the molecular decay. Note also that harmonic measure techniques are forbidden to us as we work with systems. The way it
works is that we use in fact the dual formulation of the inequalities \eqref{R} or \eqref{N} when $p>1$. We were therefore led to investigate this further (Section \ref{sec:dualp>1}). The dual of \eqref{R}  is an inequality akin to the one needed to solve the adjoint Dirichlet problem in $L^{p'}$.  The dual formulation of \eqref{N} is new. What is also new is that these dual formulations do not require any assumption on the  operator, not even $t$-independence, but the ellipticity, because we use duality brackets and not integrals. {We also use  the integrated layer potential representation \eqref{green}  for solutions of the  dual system: the  
De Giorgi condition comes into play to  show that \eqref{green}  holds whenever $u|_{t=0}\in L^p$  and $\pd_{\nu_{A}}u|_{t=0} \in \dot W^{-1, p}$ for $p>2$.}

As for the Dirichlet problem, we can basically treat it with the duality principle that Regularity solvability with $L^p$ data is equivalent  to Dirichlet solvability for $L^{p'}$ data of the dual system. 
While the Regularity to Dirichlet direction has been known since \cite{KP} for real symmetric equations, the converse is fairly recent for general systems (some partial results for real symmetric equations in Lipschitz domains are in \cite{S})  {and requires to incorporate square functions in the formulation of the Dirichlet problem}. This was proved in full generality in \cite{AR} for $p=2$ and then in \cite{HKMP2} for equations and $1+n\ge 3$ and $p\ne 2$ {(Both articles allow some $t$dependence as well)}. We reprove  and strengthen it even with the hypotheses there and also extend it to $H^1$ for Regularity vs BMO (or VMO) for Dirichlet. The Dirichlet problem is stated only with a square function estimate and no non-tangential maximal control which in fact comes as \textit{a priori} information. We shall use in this part a recent result obtained by one of us together with S. Stahlhut \cite{AS}.

 We only discuss the case of the upper half-space, but of course, the analogous results hold for systems in the lower half-space. Also by change of variable, one can treat the case of special Lipschitz domains with operators that do not depend on the vertical variable. The principal example is the Laplacian for which extrapolation results were proved in the seminal paper  \cite{DaK} and pursued in \cite{Br}.  We also mention that the same strategy can certainly be developed in the unit ball with radially independent coefficients instead, that is,  the framework of \cite{KP}, {using the first order formalism developed in \cite{AR}}.  This would require writing out some details on layer potentials. We leave this to further developments. 

All our estimates in this article depend only on ellipticity constants $\|A\|_{\infty}$ and the largest $\lambda$ in the specified ellipticity inequality,  and on the constants in the De Giorgi condition when assumed.

The second author is supported by the Fondation Math\'ematique Jacques Hadamard.   
The  authors were partially supported by the ANR project ``Harmonic analysis at its boundaries'' ANR-12-BS01-0013-01 and they thank the ICMAT for hospitality during the writing of this article. 
We warmly thank S. Hofmann for  discussions pertaining to this work, for providing us with unpublished material  and letting us use some of it, and for helping us with  historical comments.

\section{General theory and energy solutions}\label{sec:general}
 
If $E(\Omega)$ is a normed space of $\C$-valued functions on a set  $\Omega$ and $F$ a normed space, then $E(\Omega;F)$ is the space of $F$-valued functions  with $\| |f|_{F}\|_{E(\Omega)}<\infty$. 

Denote points in $\ree$ by boldface letter $\bx,\by,\ldots$ and in coordinates in $\R \times \R^n$ by $(t,x)$ etc. We set $\R^{1+n}_+=(0,\infty)\times \R^n$. 
Consider the system of $m$  equations  given by
\begin{equation}  \label{eq:divform}
  \sum_{i,j=0}^n\sum_{\beta= 1}^m \pd_i\Big( A_{i,j}^{\alpha, \beta}(\bx) \pd_j u^{\beta}(\bx)\Big) =0,\qquad \alpha=1,\ldots, m
\end{equation}
in $\R^{1+n}_+$,
where $\pd_0= \tdd{}{t}$ and $\pd_i= \tdd{}{x_{i}}$ if $i=1,\ldots,n$.  For short, we write $\divv A \nabla u=0$ to mean \eqref{eq:divform}, where we always assume that the matrix  \begin{equation}   \label{eq:boundedmatrix}
  A(\bx)=(A_{i,j}^{\alpha,\beta}(\bx))_{i,j=0,\ldots, n}^{\alpha,\beta= 1,\ldots,m}\in L^\infty(\reu;\mL(\C^{m(1+n)})),
\end{equation} is bounded and measurable
 and  satisfies some ellipticity. For systems, we use several forms of ellipticity. One is the G\aa rding inequality
\begin{equation}   \label{eq:garding0}
   {}\int_{\reu} \re (A(\bx)\nabla g(\bx)\cdot \conj{\nabla g(\bx)}) \,  d\bx \ge \lambda 
   \sum_{i=0}^n\sum_{\alpha=1}^m {}\int_{\reu} |\partial_ig^\alpha(\bx)|^2d\bx 
\end{equation}
for all $g\in C^1_{0}(\reu;\C^m)$ ($C^1$ functions with compact support) and some $\lambda>0$, and sometimes one needs the stronger
  G\aa rding inequality
\begin{equation}   \label{eq:garding+}
   {}\int_{\reu} \re (A(\bx)\nabla g(\bx)\cdot  \conj{\nabla g(\bx)}) \,  d\bx \ge \lambda 
   \sum_{i=0}^n\sum_{\alpha=1}^m {}\int_{\reu} |\partial_ig^\alpha(\bx)|^2d\bx 
\end{equation}
for all $g\in  C^1_{0}(\ree;\C^m)$ and some $\lambda>0$.  We have set $$A(\bx)\xi\cdot \conj{\eta}= \sum_{i,j=0}^n\sum_{\alpha,\beta=1}^m A_{i,j}^{\alpha,\beta}(\bx)\xi_j^\beta \,  \conj{\eta_i^\alpha}. $$
Note that the integrals are on the upper-half space. For systems,  an elementary computation shows that  \eqref{eq:garding+} is equivalent to the G\aa rding inequality  \eqref{eq:garding}  on $\ree$ (see below)   for the 
extended matrix $A^\sharp$ to $\ree$ obtained by changing the sign when $t$ changes sign of the coefficients for the mixed $t,x_{i}$ derivatives in \eqref{eq:divform}. 
The scalar case corresponding to equations is when $m=1$. In this  case, the accretivity condition above are  equivalent to the usual pointwise accretivity 
condition 
\begin{equation}
\label{eq:pointwiseaccretivity}
\re (A(\bx)\xi\cdot \conj{\xi})\ge \lambda 
   |\xi|^2,  \ \xi \in \C^{1+n}, \ {a.e. \ on} \ \reu.
\end{equation}
Alternately, scalar can mean a diagonal system in the sense that $A_{i,j}^{\alpha,\beta}=A_{i,j}\delta ^{a,\beta}$ using the Kronecker symbol. 
When $A$ has $t$-independent coefficients,  that  is $A(t,x)=A(x)$,    \eqref{eq:garding+} is implied by 
 the strict accretivity of $A$ on  the subspace $\mH^0$  of $ L^2(\R^n;\C^{m(1+n)})$ defined by $(f_{j}^\alpha)_{j=1,\ldots,n}$ is curl free in $\R^n$ for all $\alpha$,  that is,
for some $\lambda>0$ 
\begin{equation}   \label{eq:accrassumption}
   \int_{\R^n} \re (A(x)f(x)\cdot  \conj{f(x)}) \,  dx\ge \lambda 
   \sum_{i=0}^n\sum_{\alpha=1}^m \int_{\R^n} |f_i^\alpha(x)|^2dx, \ \forall \ f\in \mH_{0}.
\end{equation}
Even when $A$ is $t$-independent, \eqref{eq:pointwiseaccretivity} is  stronger  than \eqref{eq:accrassumption}  when $m\ge 2$ except when $n=1$. See \cite{AAMc} and \cite{AR} for details.  Such conditions are stable under taking adjoint of $A$.

The system \eqref{eq:divform} is always considered in the sense of distributions with weak solutions, that is  $H^1_{loc}(\R^{1+n}_{+};\C^m)$ solutions.  

There is an important space for the theory of energy (or variational) solutions  in $\reu$.  We shall use the homogeneous space  of energy solutions 
$$\mE:=\dot H^{1}(\reu; \C^m),$$ 
which is different from the one used in \cite{KR} and \cite{HKMP2}. Recall that $\dot H^{1}(\reu)$ is   the space of $L^2_{\loc}(\reu)$ functions $u$ with finite energy  ${}\int_{\reu} |\nabla u(\bx)|^2\, d\bx <\infty$. The fact that we assume local square integrability is of particular help (in fact, it suffices to even assume $u$ to be a distribution as $u$ can then be identified with an $L^2_{loc}$ function), even if the ``norm'' is defined modulo constant. Indeed, the proof of Lemma 3.1 in \cite{AMcM} shows that $\dot H^{1}(\reu)$  imbeds into  $C([0,\infty); L^2_{\loc}(\R^n))$,  where $C(\Omega)$ stands for the space of continuous functions on $\Omega$,  (up to identification of measurable functions on null sets) and that the restriction to $\reu$ of $C^\infty_{0}(\ree)$ is dense in $\dot H^{1}(\reu)$. In particular, the trace of $\dot H^{1}(\reu)$, identifying $\pd\reu$ with $\R^n$, is the space of $f\in L^2_{loc}(\R^n)$ such that $f\in \dot H^{1/2}(\R^n)$ and has $C^\infty_{0}(\R^n)$ as dense subspace.  Thus we can interpret boundary equalities also in $L^2_{\loc}$ and measure size only with the ``homogeneous'' norm in $\dot H^{1/2}(\R^n)$.  Recall that $\dot H^{1}_{0}(\reu)$ is the subspace of $\dot H^{1}(\reu)$ consisting of functions with constant trace: it is the closure of $C^\infty_{0}(\reu)$ for  the semi-norm above. Note also that $\dot H^{1}(\reu)$ is  stable (as a set)  under multiplication by $C^\infty_{0}(\ree)$ functions restricted to $\reu$.

 In what follows, we denote by $ \dot H^{s}(\R^n)$ the homogeneous Sobolev space with exponent $s\in \R$ defined as the completion of  $L^2(\R^n)$ for the semi-norm $\|(-\Delta)^{s/2}f\|_{2}$, where $\Delta$ is the self-adjoint Laplace operator on $L^2(\R^n)$. For $s>0$, it is the closure of $C^\infty_{0}(\R^n)$ ($=$ limits of Cauchy sequences for the homogeneous semi-norms) and can be realized  as a subset of $L^2_{loc}(\R^n)$,  and it becomes a Banach space when moding out polynomials of some order. For $s<0$,   it is a space of tempered distributions, identified with the dual of $ \dot H^{-s}(\R^n)$ in the usual  sesquilinear pairing. 
It is convenient to introduce the space $\dot H^{s}_{\nabla}(\R^n; \C^n):= \nabla \dot H^{s+1}(\R^n) = \nabla (-\Delta)^{-1/2} \dot H^s(\R^n)$.  Using the boundedness of the Riesz transforms on $\dot H^s(\R^n)$ for all  $s\in \R$,  it is  the subspace of  the curl free elements in $\dot H^{s}(\R^n;\C^n)$. We will use them for  $s=-1/2$ and $s=1/2$, in which case they are dual spaces for the usual duality  and notice that they both have $\mD_{\nabla}(\R^n;\C^n):= \nabla C^\infty_{0}(\R^n)$ as a dense subspace.

Let us continue with the definition of the conormal derivative and the abstract Green's formula. 

\begin{lem}\label{lemma:abstract Green} 
Let $A(\bx)$ be any bounded measurable matrix in $\reu$. Let $u\in \mE$ such that $\divv A  \nabla u=0$  in $\reu$. Then, there exists a distribution  in $\dot H^{-1/2}(\R^n; \C^m)$, denoted by $\pd_{\nu_{A}}u|_{t=0}$ or $\pd_{\nu_{A}}u_{0}$\footnote{We shall use both notations.} and called the conormal derivative of $u$ at the boundary, such that  for any $\phi\in \mE$ with $\phi_{0}= \phi|_{t=0}$, 
\begin{equation}
\label{eq:conormal}  {}\int_{\reu} A\nabla u \cdot \overline{\nabla \phi} = - \langle \pd_{\nu_{A}}u_{0}, \phi_{0}\rangle.
\end{equation}
 In particular, for  $u,w \in \mE$ with  $\divv A  \nabla u=0$  and $\divv A^*  \nabla w=0$  in $\reu$, one has the abstract Green's formula 
\begin{equation}
\label{eq:green}
\langle u_{0}, \pd_{\nu_{A^*}}w_{0}\rangle = \langle \pd_{\nu_{A}}u_{0}, w_{0}\rangle.
\end{equation}
\end{lem}

The brackets are interpreted in the $\dot H^{-1/2}, \dot H^{1/2}$ sesquilinear duality, but, by abuse, in no definite order for the factors so as to make the formula look like  the Green's formula obtained by integration by parts (when feasible).   The conormal derivative  agrees with ${\nu}\cdot (A\nabla u)|_{t=0}$ whenever this makes sense, where $\nu$ is the upward unit vector in the $t$-direction  (hence the inward normal for $\reu$).  This convention for conormal derivatives will be useful later. This explains the negative sign in the defining formula. 

\begin{proof} The definition of the conormal derivative is  a consequence of the facts that \eqref{eq:conormal} is 0 when $\phi \in \mE$ with constant trace (because $C^\infty_{0}(\reu; \C^m)$ is dense in it) and that the trace is bounded from $\mE$ onto $\dot H^{1/2}(\R^n)$. The details are left to the reader. 
The abstract Green's formula  follows immediately from definition of the conormal derivatives. 
\end{proof}

Remark that the theory of energy solutions (that is, solutions of  $\divv A \nabla u=0$ in $\mE$)   done in \cite[Section 3]{AMcM} for $t$-independent systems satisfying  \eqref{eq:accrassumption} extends immediately to $t$-dependent systems satisfying the appropriate G\aa rding inequality allowing to use the Lax-Milgram lemma.  We state the well-posedness results  for convenience. 
Note that by density,    \eqref{eq:garding0} and  \eqref{eq:garding+} extend to all $g$ in  $\dot H^{1}_{0}(\reu; \C^m)$ and $\dot H^{1}(\reu; \C^m)=\mE$ respectively.

\begin{lem}\label{lemmaN} Let  $A(\bx)$ be bounded measurable  with the stronger G\aa rding inequality \eqref{eq:garding+}. Let $g\in  \dot H^{-1/2}(\R^n;\C^m)$. Then, there is an energy solution $u\in \mE$, unique modulo constants in $\C^m$, of the system $\divv A \nabla u=0$ in $\reu$ with $\pd_{\nu_{A}}u|_{t=0}=g$   in $\dot H^{-1/2}(\R^n;\C^m)$.    
\end{lem}

This uses the Lax-Milgram lemma in $\mE/\C^m$. One can define the \emph{Neumann to Dirichlet operator} as the bounded linear operator $$\Gamma_{ND}: \dot H^{-1/2}(\R^n;\C^m) \to \dot H^{-1/2}_{\nabla}(\R^n; (\C^m)^n)$$
in such a way that  $\Gamma_{ND}(\pd_{\nu_{A}}u|_{t=0})=\nabla_{x}u|_{t=0}$, if $u$ is one of the energy solution with given Neumann datum $ \pd_{\nu_{A}}u|_{t=0}$.

\begin{lem}\label{lemma2} Let  $A(\bx)$ be bounded measurable with the G\aa rding inequality \eqref{eq:garding0}. Let $f\in L^2_{\loc}(\R^n;\C^m) \cap \dot H^{1/2}(\R^n;\C^m)$. Then, there is a unique energy solution $u\in \mE$ of the equation $\divv A \nabla u=0$ where $u|_{t=0}=f$ holds  in $\dot H^{1/2}(\R^n;\C^m) \cap L^2_{\loc}(\R^n;\C^m)$.   
\end{lem}

\begin{proof}  Given an extension $\phi$ of $f$ in $\mE$,
by the Lax-Milgram lemma applied in $\dot H^{1}_{0}(\reu; \C^m)$, there exists, unique modulo $\C^m$, a  solution $w\in \mE$ to $\divv A \nabla w=-\divv A \nabla \phi$ with $w|_{t=0}=0$ with equality in $\dot H^{1/2}(\R^n;\C^m)$. Thus $u=w+\phi$ solves $\divv A \nabla u = 0$ with $u|_{t=0}=f$. Since $f\in L^2_{\loc}(\R^n;\C^m)$, we can fix the constant by imposing the equality  in $L^2_{\loc}(\R^n;\C^m)$. Thus $u$ is uniquely defined. \end{proof}

Similarly, one can define the  \emph{Dirichlet to Neumann operator} as the bounded linear operator $$\Gamma_{DN}:  \dot H^{-1/2}_{\nabla}(\R^n; (\C^m)^n) \to \dot H^{-1/2}(\R^n;\C^m) $$ 
in such a way that  $\Gamma_{DN}(\nabla_{x}u|_{t=0})=\pd_{\nu_{A}}u|_{t=0}$, if $u$ is the energy solution with given Dirichlet datum $u|_{t=0}$ (or alternately, any of the energy solution  with given regularity datum $ \nabla_{x}u|_{t=0}$).

Let us come to some local inequalities. 
We use the notation $B(\bx,r)$ to denote the open ball in $\R^{n+1}$, centred at $\bx$, of radius $r$.
Given such a ball $B=B(\bx,r)$, we let 
$\kappa B$ denote the concentric dilate of $B$ by a factor of $\kappa$.
For $x\in\R^n$, we let $\Delta=\Delta(x,r):= B((0,x),r)\cap (\{0\} \times \R^n)$ denote the 
``surface ball" on $\R^n$ centred at $x$ and with radius $r$ and $B_{+}(x,r)= B((0,x), r)\cap \reu$ the half-ball. 

If
 $A(\bx)$ is bounded measurable with the G\aa rding inequality \eqref{eq:garding0}, any weak solution $u$ in a ball $B=B(\bx,r)$ with $ B \subset \reu$ of  $\divv A \nabla u=0$ enjoys the Caccioppoli inequality for any $0<\alpha<\beta<1$  and some $C$ depending on the ellipticity constants, $n,m$,  $\alpha$ and $\beta$,  
 \begin{equation}
\label{eq:caccio}
{}\int_{\alpha B} |\nabla u|^2 \le Cr^{-2} {}\int_{\beta B}  |u|^2,
\end{equation}
and any weak solution  $u\in W^{1,2}(B_{+};\C^m)=H^1(B_{+};\C^m)$ of  $\divv A \nabla u=0$ on $B_{+}= B_{+}(x,r)$ with $u|_{t=0}=0$ on $\Delta(x,r)$ 
enjoys the boundary Caccioppoli inequality for any $0<\alpha<\beta<1$  and some $C$ depending on the ellipticity constants, $n,m$,  $\alpha$ and $\beta$,  
 \begin{equation}
\label{eq:caccio+}
{}\int_{\alpha B_{+}} |\nabla u|^2 \le Cr^{-2} {}\int_{\beta B_{+}}  |u|^2.
\end{equation}
If $A(\bx)$ satisfies the stronger boundary G\aa rding inequality \eqref{eq:garding+}, then any weak solution $u\in H^1(B_{+};\C^m)$ of  $\divv A \nabla u=0$ on $B_{+}= B_{+}(x,r)$ with $\pd_{\nu_{A}}u|_{t=0}=0$\footnote{It can be defined locally.} on $\Delta(x,r)$ 
enjoys the boundary Caccioppoli inequality \eqref{eq:caccio+}.  The proofs are standard. 

This gives for example the following kind of local boundary estimates. 

\begin{prop}\label{prop:localboundaryreg} Let  $A(\bx)$ be bounded measurable  with the  G\aa rding inequality \eqref{eq:garding0}. Let $u,w \in \mE$ with  $\divv A  \nabla u=0$  and $\divv A^*  \nabla w=0$  in $\reu$. Assume that $u_{0}$ is supported in a surface ball $\Delta_{0}=\Delta(x_{0}, \rho)$ and $w_{0}$ is supported in a surface ball $\Delta=\Delta(x,r)$ with $4\Delta \cap \Delta_{0}=\emptyset.$ Then
\begin{equation}
\label{eq:localboundaryreg}
|\langle \pd_{\nu_{A}}u_{0}, w_{0}\rangle| \le C  r^{-2}\left({}\int_{\Omega_{+}} |u|^2\right)^{1/2}
\left({}\int_{\Omega_{+}} |w|^2\right)^{1/2}
\end{equation}
with $\Omega_{+}= 3B_{+}\setminus 2B_{+}$, $B_{+}= B_{+}(x,r)$. 
\end{prop}
 
 \begin{proof} Fix $\varphi\in C^\infty_{0}(\R^n)$  supported in $3B$, $\varphi=1$ on $2B$ where $B=B((0,x),r)$ with $\|\varphi\|_{\infty}\le 1$ and $\|\nabla \varphi\|_{\infty}\lesssim r^{-1}$. Let $\varphi_{0}$ be its restriction to $\R^n$. 
 Remark that $\varphi w\in \mE$ so its trace $\varphi_{0}w_{0}$ is well-defined in $L^2_{\loc}\cap \dot H^{1/2}$ and $\varphi_{0}w_{0}=w_{0}$ using that $\varphi_{0}=1$ on the support of $w_{0}$.
 Thus $\langle \pd_{\nu_{A}}u_{0}, w_{0}\rangle= \langle  \pd_{\nu_{A}}u_{0}, \varphi_{0}w_{0}\rangle$. Next, 
 $$
\langle \pd_{\nu_{A}}u_{0}, \varphi_{0}w_{0}\rangle= - {}\int A\nabla u\cdot \overline{\nabla (\varphi w)}=  + {}\int  u A\nabla \varphi\cdot \overline{\nabla w} - {}\int   A\nabla u\cdot \overline{w}\nabla \varphi
 $$
 where the last equality uses the fact that $\divv A^*  \nabla w=0$ and that $\varphi u\in \mE$ with $\varphi_{0} u_{0}=0$ so that $-{}\int   A\nabla (\varphi u)\cdot \overline{\nabla w}= \langle \varphi_{0}u_{0}, \pd_{\nu_{A^*}}w_{0}\rangle=0$. We conclude for  both terms by using Cauchy-Schwarz inequality, Caccioppoli and boundary Caccioppoli inequalities, and  that the support of  $\nabla \varphi$ is contained in $\Omega_{+}$.  
\end{proof}

\begin{rem}\label{rem:poincarereg} There are variants for the right hand side. As   $u$ vanishes on $3\Delta\setminus 2\Delta$, one can show
$r^{-1}\left({}\int_{\Omega_{+}} |u|^2\right)^{1/2} \lesssim \left({}\int_{\Omega_{+}} |\nabla u|^2\right)^{1/2}$ by using variants of Poincar\'e's inequality. The similar observation applies to $w$. 
 \end{rem}
 
There is a similar statement  for disjointly supported conormal derivatives. 

\begin{prop}\label{prop:localboundaryneu} Let  $A(\bx)$ be bounded measurable  with the stronger G\aa rding inequality \eqref{eq:garding+}. Let $u,w \in \mE$ with  $\divv A  \nabla u=0$  and $\divv A^*  \nabla w=0$  in $\reu$. Assume that $\pd_{\nu_{A}}u_{0}$ is supported in a surface ball $\Delta_{0}=\Delta(x_{0}, \rho)$ and $\pd_{\nu_{A^*}}w_{0}$ is supported in a surface ball $\Delta=\Delta(x,r)$ with $4\Delta \cap \Delta_{0}=\emptyset.$ Then
\begin{equation}
\label{eq:localboundaryneu}
|\langle u_{0}, \pd_{\nu_{A^*}}w_{0}\rangle| \le C  r^{-2}\left({}\int_{\Omega_{+}} |u|^2\right)^{1/2}
\left({}\int_{\Omega_{+}} |w|^2\right)^{1/2}
\end{equation}
with $\Omega_{+}= 3B_{+}\setminus 2B_{+}$, $B_{+}= B_{+}(x,r)$. 
\end{prop}

 \begin{proof} Let $\varphi\in C^\infty_{0}(\R^n)$  be as above. 
 Again $\varphi u\in \mE$ so its trace $\varphi_{0}u_{0}$ is well-defined in $L^2_{\loc}\cap \dot H^{1/2}$.
 Thus  $\langle u_{0}, \pd_{\nu_{A^*}}w_{0}\rangle= \langle \varphi_{0}u_{0}, \pd_{\nu_{A^*}}w_{0}\rangle$ using that $\varphi_{0}=1$ on the support of  $\pd_{\nu_{A^*}}w_{0}$.  We conclude exactly as in the previous argument. We skip details.
\end{proof}

\begin{rem} Note that  one can replace $u$ by $u-c$ in this argument as they have the same conormal derivative. Thus one can choose $c$ to our like. For example, if we choose the solution $u$ whose  average  on $\Omega_{+}$ equals 0, then  $r^{-1}\left({}\int_{\Omega_{+}} |u|^2\right)^{1/2} \lesssim \left({}\int_{\Omega_{+}} |\nabla u|^2\right)^{1/2}$ by Poincar\'e's inequality. One can do similarly with $w$. In our applications, we shall need decay estimates for $u$ if $\Delta_{0}$ and $\Delta$ are far apart and some control on $w$. See Theorem \ref{thm:extraneu}.
\end{rem}

\section{Rellich estimates and duality principles for $1< p<\infty$}\label{sec:dualp>1}

We next want to shed a new light on duality principles for global boundary estimates of Rellich type. Recall that we will not assume anything but ellipticity on the coefficients at this point.

For $1<p<\infty$, let $\dot W^{1,p}(\R^n)= \{f\in L^1_{\loc}(\R^n)\, ; \, \nabla f\in L^p(\R^n;\C^n)\}$ (one can show that this is the same space, upon identification, assuming instead $f\in \mD'(\R^n)$) and set $\|f\|_{\dot W^{1,p}}=\|\nabla f\|_{p}$. For $p=2$, this is also $\dot H^1(\R^n)$. Some well-known properties  are summarized here.

\begin{prop}\label{prop:w1p}
\begin{enumerate}
  \item $C^\infty_{0}(\R^n)$ is dense in $\dot W^{1,p}(\R^n)$. 
  \item $\dot W^{-1,p'}(\R^n)$, the dual of $ \dot W^{1,p}(\R^n)$, is the space of distributions $\divv {}{g}$ for some ${}{g} \in L^{p'}(\R^n;\C^n)$ with norm $\inf \|{}{g}\|_{p'}$ taken oven all  choices of ${}{g}$.  
\end{enumerate}
\end{prop}

We note the importance of the $L^1_{\loc}$ requirement to get the density. The following well-known lemma will be useful. 

\begin{lem}\label{lemma:inf}
For $ f\in L^1_{\loc}(\R^n)$, $\|\nabla f\|_{\dot W^{-1,p'}} \sim \inf\{ \|f+c\|_{p'}\, ; \, c\in \C\}.$
\end{lem}

 The left hand side is the norm in $\dot W^{-1,p'}(\R^n  ;  \C^n)$. In other words, the left hand side is finite if and only if there exists one (and only one since constants are not in $L^{p'}(\R^n)$) $c\in \C$ such that $f+c\in L^{p'}(\R^n)$. 

As we identify $\pd\reu$ with $\R^n$, we use here the subscript  0 to indicate the restriction to the boundary. Thus $\nabla u_{0}$ is short notation for $\nabla_{x}u_{0}$.

\begin{thm}\label{thm:boundaryreg}  Let  $A(\bx)$ be a bounded measurable matrix with the  G\aa rding inequality \eqref{eq:garding0}.  Let $1<p<\infty$. The following are equivalent.  
\begin{enumerate}
  \item There exists $C_{p}<\infty$ such that for any $u\in \mE$ solution of $\divv A  \nabla u=0$, $\|\pd_{\nu_{A}}u_{0}\|_{p}\le C_{p}\|\nabla u_{0}\|_{p}$.
  
  \item  There exists $C_{p'}<\infty$ such that for any $w\in \mE$ solution of $\divv A^*  \nabla w=0$, $\|\pd_{\nu_{A^*}}w_{0}\|_{\dot W^{-1,p'}} \le C_{p'}\|\nabla w_{0}\|_{\dot W^{-1,p'}}$.
   
\end{enumerate}
\end{thm}

\begin{thm}\label{thm:boundaryneu}  Let  $A(\bx)$ be a bounded measurable matrix with the stronger G\aa rding inequality \eqref{eq:garding+}. Let $1<p<\infty$. The following are equivalent.  
 \begin{enumerate}
  \item There exists $C_{p}<\infty$ such that for any $u\in \mE$ solution of $\divv A  \nabla u=0$, $\|\nabla u_{0}\|_{p} \le C_{p}\|\pd_{\nu_{A}}u_{0}\|_{p}$.
  
  \item  There exists $C_{p'}<\infty$ such that for any $w\in \mE$ solution of $\divv A^*  \nabla w=0$, $\|\nabla w_{0}\|_{\dot W^{-1,p'}} \le C_{p'}\|\pd_{\nu_{A^*}}w_{0}\|_{\dot W^{-1,p'}}$.
   \end{enumerate}
   
\end{thm}

Some remarks are necessary. The tangential gradient and conormal derivative at the boundary of  an energy solution are distributions in $\R^n$ (in $\dot H^{-1/2}$). Thus, finiteness of any of the norms above means that the distribution is identified with an element in the considered space  which is also  embedded in the space of distributions.  Theorem \ref{thm:boundaryreg} concerns boundary inequalities needed for solving the regularity problem for $\divv A  \nabla$ in $L^p$ and the Dirichlet problem for $\divv A^*  \nabla w=0$ in $L^{p'}$, or rather a regularity problem in $\dot W^{-1,p'}$. For $p=2$, this is akin to a result of \cite{AR}.  It  can be compared with Theorem 3.1 of \cite{HKMP2}, stated only for $t$-independent equations with De Giorgi condition and a restriction on $p$. In contrast,  our result here is independent of any kind of interior control on solutions besides the energy estimate and this is why it holds for any $p$.  The energy class is used here as an existence and uniqueness class. Any other such class would do a similar job. Theorem \ref{thm:boundaryneu} is new and relates the Neumann problem for $\divv A  \nabla$ in $L^p$ to a  Neumann problem  for $\divv A^*  \nabla$ in $\dot W^{-1,p'}$, which has not been studied up to our knowledge. A related statement appears in \cite{R2} for $p=2$.

\begin{proof}[Proof of Theorem \ref{thm:boundaryreg}] Assume (1) and let $w\in \mE$ be a solution of $\divv A^*  \nabla w=0$.  Assume also $\|\nabla w_{0}\|_{\dot W^{-1,p'}}<\infty$ otherwise there is nothing to prove. By lemma \ref{lemma:inf} and the fact that  for any $c\in \C^m$, $w+c$ is also a solution with same conormal derivative as $w$, we may assume 
$\|w_{0}\|_{p'}<\infty$. By Proposition \ref{prop:w1p}, it is enough to estimate $\langle \pd_{\nu_{A^*}}w_{0}, g\rangle$ for any $g\in C_{0}^\infty(\R^n\, ;  \, \C^m)$ with $\|\nabla g\|_{p}\le 1$. Let $u\in \mE$ be the solution of $\divv A  \nabla u=0$ with $u_{0}=g$ (Lemma \ref{lemma2}). By Lemma \ref{lemma:abstract Green},    $\langle \pd_{\nu_{A^*}}w_{0}, g\rangle= \langle w_{0},  \pd_{\nu_{A}}u_{0}\rangle$. Now $w_{0}\in L^{p'}$ and by (1), $\|\pd_{\nu_{A}}u_{0}\|_{p}\le  C_{p}\|\nabla g\|_{p}\le C_{p}$. Hence,  reinterpreting the last bracket in the usual $L^{p'}, L^p$ duality and using H\"older's inequality,  we obtain
$$
|\langle \pd_{\nu_{A^*}}w_{0}, g\rangle| \le \|w_{0}\|_{p'}\|\pd_{\nu_{A}}u_{0}\|_{p}\le  C_{p} \|w_{0}\|_{p'}
$$
 and we conclude for (2).

Conversely assume (2) and let $u\in \mE$ solution of $\divv A  \nabla u=0$. Assume also $u_{0}\in \dot W^{1,p}$ and $\|\nabla u_{0}\|_{p}<\infty$ otherwise there is nothing to prove. It is enough to
estimate $\langle \pd_{\nu_{A}}u_{0}, g\rangle$ for any $g\in C_{0}^\infty(\R^n\, ; \,  \C^m)$ with $\| g\|_{p'}\le 1$. Let $w\in \mE$ be the solution of $\divv A^*  \nabla w=0$ with $w_{0}=g$ (Lemma \ref{lemma2}).  By Lemma \ref{lemma:abstract Green}, $\langle \pd_{\nu_{A}}u_{0}, g\rangle= \langle u_{0},  \pd_{\nu_{A^*}}w_{0}\rangle$. Now $u_{0}\in \dot W^{1,p}$ and,  using (2) and  Lemma \ref{lemma:inf},
$\|\pd_{\nu_{A^*}}w_{0}\|_{\dot W^{-1,p'}}  \le C_{p'} \|\nabla g\|_{\dot W^{-1,p'}} \lesssim \| g\|_{{p'}} \le 1$. Thus reinterpreting the last bracket in the $\dot W^{1,p}, \dot W^{-1,p'}$ duality, we obtain
$$
|\langle \pd_{\nu_{A}}u_{0}, g\rangle|  \le \|u_{0}\|_{\dot W^{1,p}}\|\pd_{\nu_{A^*}}w_{0}\|_{\dot W^{-1,p'}}  \lesssim \|\nabla u_{0}\|_{p}$$
and we conclude for (1) by density.
\end{proof}

\begin{proof}[Proof of Theorem \ref{thm:boundaryneu}]  Assume (1) and let $w\in \mE$ be a solution of $\divv A^*  \nabla w=0$. Assume also $\|\pd_{\nu_{A^*}}w_{0}\|_{\dot W^{-1,p'}}<\infty$ otherwise there is nothing to prove. By Proposition \ref{prop:w1p} (in a vector-valued form), 
it is enough to estimate $\langle \nabla w_{0}, {}{g}\rangle$ for any ${}{g}\in C_{0}^\infty(\R^n\, ; \, (\C^m)^n)$ with  $\|{}{g}\|_{\dot W^{1,p}}=  \|\nabla {}{g}\|_{p}\le 1$. 
Let $u\in \mE$ be a solution of $\divv A  \nabla u=0$ with $\pd_{\nu_{A}}u_{0}= - \divv {}{g}$ (Lemma \ref{lemmaN}). By Lemma \ref{lemma:abstract Green}, 
$\langle \nabla w_{0}, {}{g}\rangle = \langle  w_{0}, - \divv{}{g}\rangle = 
\langle \pd_{\nu_{A^*}}w_{0}, u_{0}\rangle$. By (1),  $\|u_{0}\|_{\dot W^{1,p}} \le C_{p}\| \pd_{\nu_{A}}u_{0}\|_{p} \le C_{p}\|\divv {}{g}\|_{p}\lesssim 1$. Hence,  reinterpreting the last bracket in the $\dot W^{-1,p'}, \dot W^{1,p}$ duality we obtain
$$
|\langle \nabla w_{0}, {}{g}\rangle| \le \|\pd_{\nu_{A^*}}w_{0}\|_{\dot W^{-1,p'}}\|u_{0}\|_{\dot W^{1,p}} \lesssim  \|\pd_{\nu_{A^*}}w_{0}\|_{\dot W^{-1,p'}}$$ and we conclude for (2). 

Conversely, assume (2) and let $u\in \mE$ solution of $\divv A  \nabla u=0$. Assume also $\|\pd_{\nu_{A}}u_{0}\|_{p}<\infty$ otherwise there is nothing to prove.  It is enough to
estimate $\langle \nabla u_{0}, {}{g}\rangle$  for any ${}{g}\in C_{0}^\infty(\R^n\, ; \, (\C^m)^n)$ with $\| {}{g}\|_{p'}\le 1$. Let $w\in \mE$ be a solution of $\divv A^*  \nabla w=0$ with $\pd_{\nu_{A^*}}w_{0} = - \divv {}{g}$ (Lemma \ref{lemmaN}).
By (2), any such $w$ satisfies $\|\nabla w_{0}\|_{\dot W^{-1,p'}} \le C_{p'}\|\pd_{\nu_{A^*}}w_{0}\|_{\dot W^{-1,p'}} \le   
C_{p'}\|\divv {}{g}\|_{\dot W^{-1,p'}} \le C_{p'}\|{}{g}\|_{p'}\le C_{p'}$. By Lemma \ref{lemma:inf}, there exists $c\in \C^m$ such that $w_{0}+c \in L^{p'}$ with $
\|w_{0}+c\|_{p'}\lesssim \|\nabla w_{0}\|_{\dot W^{-1,p'}}$. Since $w+c$ is also a solution of the same problem, we may select $w$ by imposing $w_{0}\in L^{p'}$ which we do.    By Lemma \ref{lemma:abstract Green}, $\langle \nabla u_{0}, {}{g}\rangle = \langle  u_{0}, -\divv {}{g}\rangle=\langle  \pd_{\nu_{A}} u_{0}, w_{0}\rangle$.   As $w_{0}\in L^{p'}$ and  $\|\pd_{\nu_{A}}u_{0}\|_{p}<\infty$, it follows by reinterpreting the last bracket in the $L^p, L^{p'}$ duality that 
$$
|\langle \nabla u_{0}, {}{g}\rangle| \le C_{p'}\|\pd_{\nu_{A}} u_{0}\|_{p} \|w_{0}\|_{{p'}} \lesssim \|\pd_{\nu_{A}} u_{0}\|_{p} 
$$
and we conclude for (1). 
\end{proof}

A consequence of the proofs is the following self-improvement  of each of the 4 boundary inequalities in the above statements.

 We say  that an energy solution of $\divv A \nabla u=0$ has smooth Dirichlet data if  $u_{0}\in C_{0}^\infty(\R^n;\C^m)$ and has smooth Neumann data whenever $\pd_{\nu_{A}}u_{0}\in C_{0}^\infty(\R^n;\C^m)$ (necessarily with mean value 0).
 
\begin{thm}\label{thm:boundaryregimprov}  Let  $A(\bx)$ be a bounded measurable matrix with the  G\aa rding inequality \eqref{eq:garding0}.  Let $1<p<\infty$.  The following holds.  
\begin{enumerate}
  \item[(i)] If there exists $C_{p}<\infty$ such that for any energy solution $u$ of $\divv A  \nabla u=0$ with smooth Dirichlet data, one has $\|\pd_{\nu_{A}}u_{0}\|_{p}\le C_{p}\|\nabla u_{0}\|_{p}$, then this holds for any energy solution $u$ of $\divv A  \nabla u=0$, possibly with a different constant.
  
  \item[(ii)] If there exists $C_{p}<\infty$ such that for any energy  solution of $\divv A  \nabla u=0$ with smooth Dirichlet data one has $\|\pd_{\nu_{A}}u_{0}\|_{\dot W^{-1,p}} \le C_{p}\|\nabla u_{0}\|_{\dot W^{-1,p}}$, then this holds for any energy solution $u$ of $\divv A  \nabla u=0$, possibly with a different constant.
   \end{enumerate}
\end{thm}

\begin{proof} For (i), we remark that to prove (1) implies (2) in Theorem \ref{thm:boundaryreg}, we use  (1)  with smooth data. Thus the assumption of (i) implies (2) in Theorem \ref{thm:boundaryreg} and we conclude using the converse (2) implies (1) in the same theorem. 
The proof of (ii) is similar starting from (2)  for $A$ and $p$ instead of  $A^*$ and $p'$ in Theorem \ref{thm:boundaryreg}.
\end{proof}

For Neumann problems we have, 

\begin{thm}\label{thm:boundaryneuimprov}  Let  $A(\bx)$ be a bounded measurable matrix with the stronger  G\aa rding inequality \eqref{eq:garding+}.  Let $1<p<\infty$. The following holds.  
\begin{enumerate}
  \item[(i)] If there exists $C_{p}<\infty$ such that for any energy solution $u$ of $\divv A  \nabla u=0$, with smooth  Neumann data  one has $\|\nabla u_{0}\|_{p} \le C_{p}\|\pd_{\nu_{A}}u_{0}\|_{p}$, then this holds for any energy solution $u$ of $\divv A  \nabla u=0$, possibly with a different constant.
  
  \item[(ii)]  If there exists $C_{p}<\infty$ such that for any energy  solution of $\divv A  \nabla u=0$ with smooth Neumann data one has $\|\nabla u_{0}\|_{\dot W^{-1,p}} \le C_{p'}\|\pd_{\nu_{A}}u_{0}\|_{\dot W^{-1,p}}$, then this holds for any energy solution $u$ of $\divv A  \nabla u=0$, possibly with a different constant.
   \end{enumerate}
\end{thm}

The proof is similar noting that we use smooth data of the form $-\divv {}{g}$ in the arguments. Details are left to the reader. 

\section{Rellich estimates: the case $\frac{n}{n+1}<p\le 1$}\label{sec:rellichp<1}

Here, the duality equivalence is a subtle issue for $p<1$ but  remains for $p=1$. We prove this first. Then we consider the problem of extension from estimates on atoms to global estimates.

Let $H^p(\R^n)$  denote the real Hardy space if $\frac{n}{n+1}<p\le 1$.  We have that $H^p(\R^n)$ are distributions spaces and, in this range, $C^\infty_{0}(\R^n)$ functions with mean value 0 form a dense subspace. For $\frac{n}{n+1}<p\le 1$, let $\dot H^{1,p}(\R^n)=\{f\in \mS'(\R^n) ; \pd_{x_{i}} f \in H^p(\R^n), i=1, \ldots, n\}$ with norm 
$\|f\|_{\dot H^{1,p}(\R^n)}= \|\nabla f \|_{H^p(\R^n; \C^n)}$. This  is the homogeneous Hardy-Sobolev space which has been studied in many places (\cite{Str}, \cite{Mi}, \cite{ART}, \cite{BB}, \cite{BG}, \cite{KS}, \cite{LMc}... ). In particular,  elements in these spaces are known to be locally integrable functions and $ C_{0}^\infty(\R^n)$ is a dense subspace. 

Let us turn to recalling duality. For all of them, we use the standard hermitian duality on functions, extended appropriately.  Recall that if $\alpha=n(1/p-1)\in [0,1)$, the dual 
 of $H^p(\R^n)$ is identified with $\dot \Lambda^0(\R^n):=BMO(\R^n)$ is $p=1$ and with the homogeneous H\"older space $\dot \Lambda^\alpha(\R^n)$ of those continuous functions such that $|u(x)-u(y)|\le C|x-y|^\alpha$ for all $x,y\in \R^n$, the smallest $C$ defining the semi-norm. These spaces can also be seen within $\mD_{0}'(\R^n)$, the space of distributions modulo constants, in which they are Banach. Recall also that $H^1(\R^n)$ is the dual space of VMO$(\R^n)$ (sometimes called CMO), the closure of $C_{0}^\infty(\R^n)$ in BMO$(\R^n)$. 
 The dual of $\dot H^{1,p}(\R^n)$ is identified with $\dot\Lambda^{\alpha-1}(\R^n)$ defined  as the 
  space of distributions $\divv f$, $f\in \dot\Lambda^{\alpha-1}(\R^n)$, equipped with the quotient norm.

Let us call $X=H^p(\R^n;\C^d)$ with $d=m$ or $mn$ indifferently. Let $Y=\dot \Lambda^\alpha$ be the dual space and $ Y^{-1}=\dot \Lambda^{\alpha-1}$.

First we complete Theorems \ref{thm:boundaryreg} and \ref{thm:boundaryneu}  by the following results.

\begin{thm}\label{thm:boundaryreg2}  Let  $A(\bx)$ be a bounded measurable matrix with the  G\aa rding inequality \eqref{eq:garding0}.  Let $\frac{n}{n+1}<p \le 1$, $0\le \alpha=n( \frac{1}{p}-1)<1$, and $X$ and $Y^{-1}$ be the corresponding boundary spaces. Then (1) implies (2), where  
\begin{enumerate}
  \item There exists $C_{X}<\infty$ such that for any $u\in \mE$ solution of $\divv A  \nabla u=0$, $\|\pd_{\nu_{A}}u_{0}\|_{X}\le C_{X}\|\nabla u_{0}\|_{X}$.
  
  \item  There exists $C_{Y^{-1}}<\infty$ such that for any $w\in \mE$ solution of $\divv A^*  \nabla w=0$, $\|\pd_{\nu_{A^*}}w_{0}\|_{Y^{-1}} \le C_{Y^{-1}}\|\nabla w_{0}\|_{Y^{-1}}$.
   \end{enumerate}
   The converse holds in the case $p=1$. 
\end{thm}

\begin{thm}\label{thm:boundaryneu2}  Let  $A(\bx)$ be a bounded measurable matrix with the stronger G\aa rding inequality \eqref{eq:garding+}. 
Let $\frac{n}{n+1}<p \le 1$, $0\le \alpha=n( \frac{1}{p}-1)<1$, and $X$ and $Y^{-1}$ be the corresponding boundary spaces. Then (1) implies (2), where 

\begin{enumerate}
  \item There exists $C_{X}<\infty$ such that for any $u\in \mE$ solution of $\divv A  \nabla u=0$, $\|\nabla u_{0}\|_{X} \le C_{X}\|\pd_{\nu_{A}}u_{0}\|_{X}$.
  
  \item  There exists $C_{Y^{-1}}<\infty$ such that for any $w\in \mE$ solution of $\divv A^*  \nabla w=0$, $\|\nabla w_{0}\|_{Y^{-1}} \le C_{Y^{-1}}\|\pd_{\nu_{A^*}}w_{0}\|_{Y^{-1}}$.
   \end{enumerate}
   The converse holds if $p=1$.
\end{thm}

The proofs are \textit{mutatis mutandi} the same as when $1<p<\infty$ using $C^\infty_{0}$ functions  with mean value 0 as test functions in $H^p$. The converses at $p=1$ use the fact that $H^1$ is the dual space of $VMO$ in which test functions are dense and also that $\|\nabla f\|_{BMO^{-1}}\sim \|f\|_{BMO}$ for $f\in L^1_{loc}$. Details are left to the reader. 

We now turn to the extension problem. Recall that a 2-atom for  $H^p(\R^n)$ is a  function $a\in L^2(\R^n)$ such that 
 \begin{enumerate}
\item the support of $a$ is contained in a ball $\Delta(x_{0},r)$, 
  \item  $\|a\|_{2}\le r^{-n(1/p-1/2)}$,
  \item $\int a =0$.
\end{enumerate} 
A  2-atom for $H^p(\R^n)$ is smooth if it is $C_{0}^{\infty}(\R^n)$.  Set $\mD_{0}(\R^n)$ the subspace of  $C^\infty_{0}(\R^n)$ of functions with mean 0.  For our purpose here, observe that 
2-atoms for $H^p(\R^n)$ are elements of $\dot H^{-1/2}(\R^n)$. In fact, if $a$ is such a function, a classical result of  Ne\u{c}as \cite{N} asserts that there exists a function $b\in W^{1,2}(\R^n;\C^n)$  (inhomogeneous Sobolev space)   with support in the ball supporting $a$ such that $a=\divv b$ on $\R^n$. Thus, if $f\in \dot H^{1/2}(\R^n)$, $\langle a,f \rangle =  -\langle b,\nabla f \rangle$ and we remark that  by interpolation $\|b\|_{\dot H^{1/2}(\R^n,\C^n)}\le C(\|b\|_{2}\|\nabla b\|_{2})^{1/2} <\infty$, while $\nabla f \in  \dot H^{-1/2}(\R^n;\C^n)$.

 Let 
$$H^p_{\nabla}(\R^n;\C^n)= \{{g}\in H^p(\R^n; \C^n); \curl\,  {g}=0\} = \{\nabla f; f\in \dot H^{1,p}(\R^n)\}$$ 
and $\mD_{\nabla}(\R^n;\C^n):=\nabla (C^\infty_{0}(\R^n))$. It is easy to see using $\dot H^{1,p}$ spaces that $\mD_{\nabla}(\R^n;\C^n)$  is dense in $H^p_{\nabla}(\R^n;\C^n)$.
As for the duality, one can see that the dual (for the same duality as the other spaces) of $H^p_{\nabla}(\R^n;\C^n)$ is  
$\dot \Lambda^\alpha_{\nabla}(\R^n;\C^n)$ identified as the subspace of $\dot \Lambda^\alpha(\R^n;\C^n)$ with curl free elements.  The identification is easy.  For the duality, if  $\mR=\nabla (-\Delta)^{-1/2}$ is the array of Riesz transforms, then the self-adjoint operator $\mR\mR^*$ extends to  a bounded projection from $H^p(\R^n;\C^n)$ onto $H^p_{\nabla}(\R^n;\C^n)$  and similarly from $\dot \Lambda^\alpha(\R^n;\C^n)$ onto $\dot \Lambda^\alpha_{\nabla}(\R^n;\C^n)$. From here, the duality for the ranges of the projection follows from that of the source spaces.  

For $H^p_{\nabla}(\R^n; \C^n)$,  the 2-atoms in \cite{LMc} for differential forms on $\R^n$, identifying $\nabla$ with the exterior derivative on functions, suit our needs. It was done for $p=1$   there (Definition 6.1),   but careful inspection shows it extends to $\frac{n}{n+1}<p \le 1$ with the following definition. 

\begin{defn}\label{defn:atom} Let $\frac{n}{n+1}<p \le 1$. A $2$-atom for $H^p_{\nabla}(\R^n;\C^n)$  is a  function $a\in L^2(\R^n;\C^n)$ such that 
 \begin{enumerate}
 \item there exists $b\in L^2(\R^n)$ such that $a=\nabla b$ in $\mD'(\R^n)$,
\item the supports of $a$ and $b$ are contained in a ball $\Delta(x_{0},r)$, 
  \item  $\|a\|_{2}\le r^{-n(1/p-1/2)}$,
  \item  $\|b\|_{2}\le r^{1-n(1/p-1/2)}$.
\end{enumerate} 
Note that 2-atoms for $H^p_{\nabla}(\R^n;\C^n)$ are in particular 2-atoms for  $H^p(\R^n;\C^n)$ since they satisfy $\int a=0$.
A  $2$-atom for $H^p_{\nabla}(\R^n;\C^n)$ is smooth when $b\in C^\infty(\R^n)$. 
\end{defn}

It is easily seen from the definition that $2$-atoms for $H^p_{\nabla}(\R^n;\C^n)$ belong to the space  
$ \dot H^{-1/2}_{\nabla}(\R^n;\C^n)$.  We shall require the following result. 

\begin{prop} 
\begin{enumerate}
  \item \subitem Let $T$ be a linear operator defined on $\mD_{0}(\R^n)$ such that 
$\sup \|Ta\|_{H^p_{\nabla}(\R^n;\C^n)} <\infty$, where the supremum is taken over all smooth 2-atoms for  $H^p(\R^n)$. Then $T$ has a bounded extension from $H^p(\R^n)$ into $H^p_{\nabla}(\R^n;\C^n)$.
 \subitem Suppose, in addition, that $T$ was originally a bounded linear operator from $\dot H^{-1/2}(\R^n)$ into $ \dot H^{-1/2}_{\nabla}(\R^n;\C^n)$. Then $T$ and the above extension coincide on $\dot H^{-1/2}(\R^n) \cap H^p(\R^n)$. 
   \item \subitem Let $T$ be a linear operator defined on $\mD_{\nabla}(\R^n;\C^n)$ such that 
$\sup \|Ta\|_{H^p(\R^n)} <\infty$, where the supremum is taken over all smooth 2-atoms for  $H^p_{\nabla}(\R^n;\C^n)$. Then $T$ has a bounded extension from $H^p_{\nabla}(\R^n;\C^n)$ into $H^p(\R^n)$.
  \subitem Suppose, in addition, that $T$ was originally a bounded linear operator from  $
  \dot H^{-1/2}_{\nabla}(\R^n;\C^n)$ into $\dot H^{-1/2}(\R^n)$. Then $T$ and the above extension coincide on $\dot H^{-1/2}_{\nabla}(\R^n;\C^n) \cap H^p_{\nabla}(\R^n;\C^n)$. 
\end{enumerate}
\end{prop}

Of course the statement applies with $\C^m$-valued functions instead of $\C$-valued functions. 
 
 \begin{proof} The first part of (1) is a special case of Theorem 1.1 in \cite{YZ}. For the second part, 
 we adapt a classical procedure found, for example as Proposition 4.2 of \cite{MSV}, which is also reminiscent of the method of proof of Theorems \ref{thm:boundaryneuimprov} and \ref{thm:boundaryregimprov}.   Call $\tilde T$ the extension defined above. 
 First if $f\in \dot H^{1/2}_{\nabla}(\R^n;\C^n) \cap \dot \Lambda^\alpha_{\nabla}(\R^n;\C^n)$ and $g\in \mD_{0}(\R^n)$,
 $$
 \langle g, T^*f  \rangle =  \langle  Tg,f  \rangle=  \langle \tilde Tg ,  f\rangle=  \langle g , \tilde T ^*f\rangle.
 $$
 The first two brackets are interpreted in the  $\dot H^{-1/2}, \dot H^{1/2}$ duality,  then we use that $Tg=\tilde Tg$ as $g$ can be seen as a multiple of a 2-atom for $H^p(\R^n)$. This allows us to reinterpret the last two brackets in the $H^p, \dot \Lambda^\alpha$ duality. We conclude that $T^*f=\tilde T ^*f$ in $\mD_{0}'(\R^n)$, hence they both belong to $\dot H^{1/2}_{\nabla}(\R^n;\C^n) \cap \dot \Lambda^\alpha_{\nabla}(\R^n;\C^n)$ and differ by a constant. Next, let  $f\in \mD_{\nabla}(\R^n; \C^n)$ (contained in   both $\dot H^{1/2}_{\nabla}(\R^n;\C^n), \dot \Lambda^\alpha_{\nabla}(\R^n;\C^n)$ and dense in the first) and $g\in\dot H^{-1/2}(\R^n) \cap H^p(\R^n)$. Then
 $$
\langle  Tg,f  \rangle= \langle g, T^*f  \rangle =  \langle g , \tilde T ^*f\rangle =    \langle \tilde Tg ,  f\rangle.
 $$
 Here, the first two brackets are interpreted in the $\dot H^{-1/2}, \dot H^{1/2}$  duality. The second can be reinterpreted in the $H^p, \dot \Lambda^\alpha$ duality. In the second equality, we then use $T^*f=\tilde T ^*f$ up to a constant, which is annihilated. In particular, we obtain that 
 $|\langle \tilde Tg ,  f\rangle| \le \|Tg\|_{\dot H^{-1/2}} \|f\|_{\dot H^{1/2}}$. Thus, $\tilde Tg \in \dot H^{-1/2}_{\nabla}(\R^n;\C^n)$  and we conclude that $\tilde T g= Tg$.
  
 The proof of  (2)   is the same, once we make the following observation. The proof of Theorem 1.1 in \cite{YZ} depends only on having a Calder\'on reproducing formula with smooth and compactly supported convolution kernels and the characterisation of the Hardy space by the Lusin functional based on the kernels involved. Now, the atomic decomposition of \cite{LMc} is exactly obtained via the same strategy   with further algebraic constraints on the kernels to obtain the gradient form of the 2-atoms. Thus the analysis in \cite{YZ} applies and their Theorem 1.1 extends to our situation.  This provides us with the extension. The second part of the argument  is \textit{mutatis mutandi} the same. 
\end{proof}

We can now state the results we are after.

\begin{thm}\label{thm:boundaryregimprovhardy}  Let  $A(\bx)$ be a bounded measurable matrix with the  G\aa rding inequality \eqref{eq:garding0}.  Let $\frac{n}{n+1}<p \le 1$.   
 If $\sup \|\pd_{\nu_{A}}u_{0}\|_{H^p(\R^n;\C^m)}\le C_{p}$ taken over all energy solutions $u$ of $\divv A  \nabla u=0$ with (smooth) 2-atoms for $H^p_{\nabla}(\R^n;(\C^m)^n)$ as regularity data, then $\|\pd_{\nu_{A}}u_{0}\|_{H^p(\R^n;\C^m)}\le C_{p}\|\nabla u_{0}\|_{H^p_{\nabla}(\R^n;(\C^m)^n)}$ for any energy solution $u$ of $\divv A  \nabla u=0$, possibly with a different constant.
  \end{thm}
  
  \begin{thm}\label{thm:boundaryneuimprovhardy}  Let  $A(\bx)$ be a bounded measurable matrix with the stronger  G\aa rding inequality \eqref{eq:garding+}.  Let $\frac{n}{n+1}<p \le 1$.  
 If $\sup \|\nabla u_{0}\|_{H^p_{\nabla}(\R^n;(\C^m)^n)}\le C_{p}$ taken over all energy solutions $u$ of $\divv A  \nabla u=0$ with (smooth) 2-atoms for $H^p(\R^n;\C^m)$ as Neumann data, then $\|\nabla u_{0}\|_{H^p_{\nabla}(\R^n;(\C^m)^n)} \le C_{p} \|\pd_{\nu_{A}}u_{0}\|_{H^p(\R^n;\C^m)} $ for any energy solution $u$ of $\divv A  \nabla u=0$, possibly with a different constant.
  \end{thm}

The proof of the first theorem follows on applying (2) of the above proposition to the Dirichlet to Neumann operator $\Gamma_{DN}$ and of the second on applying (1)  of the above proposition to the Neumann to Dirichlet operator $\Gamma_{ND}$.

\section{Fundamental solutions}

Assuming the De Giorgi condition for the operators $\divv A \nabla $ and $\divv A^* \nabla $ in $\ree$, these operators have fundamental solutions which have the expected estimates. It is convenient to  state the relevant statements and references. We use the notation of section 2 for points and balls in $\ree$. 

Consider an elliptic  system $\divv A \nabla$  in  $\R^{1+n}$, with bounded measurable matrix $A(\bx)$ depending  on all variables. Ellipticity is taken in the sense the G\aa rding inequality
\begin{equation}   \label{eq:garding}
   \int_{\R^{1+n}} \re (A(\bx)\nabla g(\bx)\cdot \conj{\nabla g(\bx)}) \,  d\bx \ge \lambda 
   \sum_{i=0}^n\sum_{\alpha=1}^m \int_{\R^{1+n}} |\partial_ig^\alpha(\bx)|^2d\bx, 
\end{equation}
for all $g \in \dot H^{1}(\ree; \C^m)$ and some $\lambda>0$.
We say that $\divv A \nabla$   satisfies the De Giorgi condition if
\begin{equation}  \label{eq:DGN}
  \int_{B(\bx, r)} |\nabla u |^2 \lesssim (r/R)^{n-1+2\mu}\int_{B(\bx, R)} |\nabla u |^2
   \end{equation}
holds for all weak solutions $u$ to $\divv A\nabla u=0$ in $B(\bx, 2R)\subset\R^{1+n}$ and all  $\bx \in \ree$ and $0<r<R$, for some $\mu\in (0,1]$.
It is known  that \eqref{eq:DGN} is equivalent to the H\"older estimate of Nash
\begin{equation}  \label{eq:Morrey}
   \esssup_{\by, \bz\in B(\bx, R), \by\ne \bz} \frac {|u(\by)-u(\bz)|}{|\by-\bz|^\alpha} \lesssim
    R^{-\alpha-(1+n)/2}\left( \int_{B(\bx, 2R)} |u|^2 \right)^{1/2} 
\end{equation}
 whenever $u$ is a weak solution to $\divv A\nabla u=0$ in $B(\bx, 3R)\subset\R^{1+n}$, for any  $\bx\in \ree$ and $0<r<R$,  for some $\alpha\in (0,1]$.  Furthermore, the upper bounds of $\mu$'s in \eqref{eq:DGN}  and $\alpha$'s in \eqref{eq:Morrey} are equal, which we set $\mu_{DG}^A$ and call the De Giorgi exponent of $\divv A \nabla$.

De Giorgi's theorem \cite{DeG} states that \eqref{eq:DGN}, or equivalently \eqref{eq:Morrey} of Nash \cite{Na}, holds for all divergence form equations ($m=1$) $\divv A \nabla u=0$ when $A$ is real.   It also holds for any system if dimension $1+n=2$ \cite{Mor}. \cite{AAAHK}, Section 11, shows it is also the case in dimension $1+n=3$ (the argument presented for equations, works for our systems as it relies  on Meyers'  \cite{Me} and Caccioppoli estimates which holds for such systems) when, in addition, $A$ has $t$-independent coefficients. Finally, in  \cite{A}, it is shown that \eqref{eq:DGN} is a stable property under $L^\infty$ perturbations of $A$ (again, this is shown for equations but it holds for our systems).

Estimates \eqref{eq:DGN} and \eqref{eq:Morrey} also imply the {\em Moser local boundedness estimate} \cite{Mo}
\begin{equation}  \label{eq:Mos}
   \esssup_{\by\in B(\bx,R)} |u(\by)| \lesssim
    R^{-(1+n)/2}\left( \int_{B(\bx, 2R)} |u|^2 \right)^{1/2} 
\end{equation}
whenever $\divv A \nabla u=0$ in $B(\bx, 3R)\subset\R^{1+n}$ for all  $\bx\in \ree$ and $0<R<\infty$.
We refer to \cite[Sec. 2]{HK} for details.

\begin{prop}\label{prop:fund} Let $n+1\ge 2$ and assume that  $\divv A\nabla $ and $\divv A^*\nabla $ satisfy the De Giorgi condition  or equivalently, the Nash local regularity condition. Then  $\divv A\nabla $ and $\divv A^*\nabla $ 
have a fundamental solution $\Gamma^{A}( \bx; \by)=\Gamma^A_{\by}(\bx) \in W^{1,1}_{loc}(\R^{1+n};\mL(\C^m))$ at pole $\by\in \R^{1+n}$ and  $\Gamma^{A^*}( \by; \bx)=\Gamma^{A^*}_{\bx}(\by) \in W^{1,1}_{loc}(\R^{1+n};\mL(\C^m))$ at pole $\bx \in \R^{1+n}$ $($ie,  $\divv_{\bx} A(\bx) \nabla_{\bx} \Gamma^A_{\by}(\bx)=  \delta _{\by}(\bx)$ and $\divv_{\by} A^*(\by) \nabla_{\by} \Gamma^{A*}_{\bx}(\by)=  \delta _{\bx}(\by))$ with for some $0<\mu<\inf (\mu_{DG}^A, \mu_{DG}^{A^*})$,
\begin{equation}
\label{eq:gammasize}
|\Gamma^A( \bx; \by)| \lesssim |\bx -\by|^{1-n}, \  \mathrm{if}\ 1+n\ge 3 , \ \mathrm{and}\  \lesssim 1+ |\ln |\bx - \by||  \  \mathrm{if}\ 1+ n=2,
\end{equation}
\begin{equation}
\label{eq:gammaholder}
|\Gamma^A( \bx; \by)-\Gamma^A( \bx; \by') | \lesssim \left( \frac{|\by-\by'|}{|\bx -\by|}\right)^\mu|\bx -\by|^{1-n}, \  \mathrm{if}\ |\by-\by'| \le  |\bx -\by|/2,
\end{equation}
 and 
 \begin{equation}
\label{eq:gammagrad}
{}\int_{B(\bz,\rho)} |\nabla_{\by}\Gamma^{A^*}( \by; \bx)|^2\, d\by \le C \frac{\rho^{n-1+2\mu}}{|\bx-\bz|^{2n-2+2\mu}} \  \mathrm{if}\ \rho>0  \ \mathrm{and}\   |\bx-\bz|\ge 2\rho,
\end{equation}
and symmetrically by exchanging the roles of $\Gamma^A$ and $\Gamma^{A^*}$. 
\end{prop}

 \begin{proof}  This result is in \cite{R1}, Theorem 1.2.  
Note that this result is stated slightly differently there but all what is used is the De Giorgi condition.
 Note also that the estimate \eqref{eq:gammaholder} is stated with an extra multiplicative log factor when $1+n=2$, but the proof there does give what we state. 
\end{proof}
 
\begin{rem}
1) If  $\rho \sim |\bx-\bz|/2$,  then the right hand side of  \eqref{eq:gammagrad}, $|\bx-\bz|^{1-n}$, is obtained during the construction. The gain $\mu$ comes from use of the  De Giorgi condition \eqref{eq:DGN} with the balls $B(\bz, \rho) \subset B(\bz, |\bz-\bx|/2)$.

2) Assume $1+n\ge 3$.  There is a previous construction  in  \cite{HK}  under the stronger pointwise ellipticity assumption on $A$. But examination shows that only \eqref{eq:garding} is required. More estimates are obtained there. These are the only ones we need here. 
In particular,  uniqueness of the fundamental solution is proved together with the symmetry relation $\Gamma^{A^*}( \by; \bx)= {\Gamma^A( \bx; \by)}^*$,  where  the latter is the hermitian adjoint of $ 
 {\Gamma^A( \bx; \by)}$ as an $m\times m $ matrix.  
 
 3) Assume $1+n=2$. The first construction for complex coefficients is in \cite{AMcT} for scalar operators ($m=1$).  An analogous estimate
was  obtained in \cite{DoK}, Theorem 2.21, for systems but was only carried out explicitly assuming strong ellipticity. See also \cite{CDoK}.
   \cite[Chapter 4]{B} used the construction in \cite{AMcT} and showed uniqueness and also that it is possible to choose the constant of integration in such a way  the symmetry relation holds. This construction extends \textit{mutatis mutandi} to systems and does give the above estimates, with possible exception of uniqueness as the argument relies on properties of harmonic functions.  
    
  \end{rem}

\section{Decay estimates for  energy solutions }\label{sec:decay}

In this section, we consider without mention  systems with  $A(\bx)$  bounded, measurable,  non necessarily $t$-independent, with the stronger G\aa rding inequality \eqref{eq:garding+} and we assume that the reflected matrix $A^\sharp$ and its adjoint satisfy the De Giorgi condition on $\ree$.\footnote{This is a way of saying that $A$ and its adjoint satisfy both interior and boundary De Giorgi condition on the upper half-space. Some variants in the hypotheses are certainly possible here.} The number $\mu>0$ in the statements below will be any number less than the De Giorgi exponents for $A^\sharp$ and its adjoint. 

This situation covers  dimension $1+n=2$ or dimensions $1+n\ge 3$ with $A$  close in $L^\infty$ to a real and scalar matrix (for systems, scalar means diagonal). In particular we cover the case of real equations. In this respect, our first result extends  Lemma 4.9 of \cite{HKMP1}.

\begin{lem}\label{lemma1}       Let $x_0\in \R^n$, $r>0$, and set $\bx_{0}:=(0,x_{0})$, 
$B:= {B(\bx_0,r)}$, $\Delta:=\Delta(x_{0},r)$. 
Suppose that $w \in L_{\loc}^2(\reu\setminus \overline{B}; \C^m)$ with $\nabla w \in L^2(\reu\setminus \overline{B}; (\C^m)^{1+n})$
is a weak solution of $\divv A \nabla w=0$ in $\reu\setminus \overline{B}$, and that $w|_{\rn\setminus \overline{\Delta}} \equiv 0$. 
Then $w$ is (identified to) a bounded and continuous function on $\overline{\reu}\setminus 3B$,
and  for some  constants $C$ and $\mu >0$, depending only upon the assumption on $A$, 
$$|w(\bx)|\leq C  
\frac{r^{\frac{n+1}{2}+\mu-2}}{|\bx-\bx_0|^{n-1+\mu}} \, \left({}\int_{\Omega_{+}} |w|^2\right)^{1/2},\qquad |\bx-\bx_0|\geq 3r\, .$$
Here, $\Omega_{+}=3B_{+}\setminus 2B_{+}$ and $B_{+}= \reu \cap B$.
In particular, $w\to 0$ at infinity. 
\end{lem}

\begin{proof}  Let us drop the dependence on $m$ in the notation to simplify the exposition.  First, the assumption $w \in L_{\loc}^2(\reu\setminus \overline{B})$ with $\nabla w \in L^2(\reu\setminus \overline{B})$ implies that $w \in C([0,\infty); L^2_{\loc}(\rn \setminus \overline{\Delta}))$. See the argument in \cite{AMcM}. In particular, the equation $w|_{\rn\setminus \overline{ \Delta}} \equiv 0$ holds in $L^2_{\loc}$. Set $v=w\chi$ where $\chi$ is a smooth real-valued function supported on $\R^{1+n}\setminus (11/5)B$, which is $1$ on $\R^{1+n}\setminus (14/5)B$ with $ \|\chi\|_{\infty}\le 1$ and  $\|\nabla \chi\|_{\infty}\le C/r$. One has that $v\in \dot H^1(\reu)$, $v|_{\rn} \equiv 0$ holds in $L^2_{\loc}$ and $\divv A \nabla v=f +\divv \mathbf{g}$ weakly in $\reu$, with $f= A \nabla \chi. \nabla w$  and $\mathbf{g}= A \nabla \chi w.$ Note that 
$$\|\mathbf{g}\|_{2} \lesssim  r^{-1} \left({}\int_{\Omega_{+}} |w|^2\right)^{1/2},$$ the implicit constant depending on the $L^\infty$ bound for $A$ and dimension.
Also
$$\|f\|_{2}\le r^{-1} \left({}\int_{\reu \cap ((14/5)B\setminus (11/5)B)} |\nabla w|^2\right)^{1/2} \lesssim 
r^{-2} \left({}\int_{\Omega_{+}} |w|^2\right)^{1/2},$$
 where the last inequality  uses boundary and interior Caccioppoli inequalities.
 
 One can represent $v$ using the method of reflection. Let $v^\sharp, f^\sharp$ be the odd extensions of $v,f$ and  ${\bf g}^\sharp$ is the extension of $\mathbf{g}$ defined by $\mathbf{g}^\sharp (\by)= N\mathbf{g}(N\by)$, with $N(t,y)=(t,y)^\sharp= (-t,y)$. Remark that since $f^\sharp \in L^2$, with support in $3B$ and mean value condition $\int f^\sharp=0$, then $f^\sharp \in \dot H^{-1}(\ree)$ with $\|f^\sharp\|_{ \dot H^{-1}(\ree)} \lesssim r\|f^\sharp\|_{2}.$ Thus $v^\sharp \in \dot H^1(\ree)$ with $\divv A^\sharp \nabla  v^\sharp= f^\sharp + \divv \mathbf{g}^\sharp$.  As $\dot H^1(\ree)$  is a uniqueness class modulo constants for this equation (since we have \eqref{eq:garding} for $A^\sharp$), it follows that $v^\sharp$ is the unique odd (with respect to $N$) element in $ \dot H^1(\ree)$ solving this equation.  If  $v^\sharp_{1}$ is the unique odd solution obtained from $f^\sharp$ and $v^\sharp_{2}$ is the unique odd solution obtained from $-\divv g^\sharp$, one has $v^\sharp=v^\sharp _{1}-v^\sharp_{2}$ (in $L^2_{loc}$). 
Now using the fundamental solution $\Gamma^{A^\sharp}$,   $v^\sharp_{1}(\bx)$ and $v^\sharp_{2}(\bx)$ have the respective   integral representations  for $\bx\in \ree$ away from the supports of $f^\sharp$ and $g^\sharp$, 
$$
v^\sharp_{1}(\bx)= 
 {}\int_{\ree} \Gamma^{A^\sharp}(\bx;\by)f^\sharp (\by)\, d\by,
$$
$$
v^\sharp_{2}(\bx)=
 {}\int_{\ree} (\nabla_{\by}\Gamma^{A^\sharp})(\bx;\by) \mathbf{g}^\sharp(\by)\, d\by.$$
One can check that changing $\bx$ to $\bx^\sharp$ change the signs of both integrals. That is, both integrals are odd with respect to $N$. It follows that $v^\sharp_{1}$ and $v^\sharp_{2}$ agree with these integrals in $L^2_{loc}$ away of the supports of $f^\sharp$ and $g^\sharp$. 
  Next,  restricting to $\bx \in \reu$, still away from the supports of $f$ and $g$, we can rewrite the integrals   as
 $$v^\sharp_{1}(\bx)  ={}\int_{\reu} (\Gamma^{A^\sharp}(\bx;\by) - \Gamma^{A^\sharp}(\bx;\by^\sharp))f(\by)\, d\by,
  $$
  $$
v^\sharp_{2}(\bx) = {}\int_{\reu} ((\nabla_{\by}\Gamma^{A^\sharp})(\bx;\by) - (\nabla_{\by}\Gamma^{A^\sharp})(\bx,\by^\sharp))\mathbf{g}(\by)\, d\by.
$$
We have shown that $v$ is the difference of these 2 integrals in $L^2_{loc} $ away from the supports of $f$ and $g$. As
 they have the desired pointwise bounds using Proposition \ref{prop:fund} applied with $A^\sharp$,  the conclusion follows from the fact that $v=w$ on the range where these pointwise inequalities hold.
\end{proof}

\begin{lem}\label{lemma3}  Let $f\in L^2(\R^n;\C^m) \cap \dot H^{1/2}(\R^n;\C^m)$ with compact support in the surface ball $\Delta={\Delta(x_{0},r)}$. Then the solution of $\divv A \nabla u=0$ where $u|_{t=0}=f$ given by Lemma \ref{lemma2} is locally H\"older continuous on $\reu$, continuous up the boundary away from $B(\bx_{0},3r)$, tends to 0 at $\infty$ and,  has the estimate for some $C,\mu>0$, 
$$|u(\bx)|\leq C   
\frac{r^{\frac{n+1}{2}+\mu-1}}{|\bx-\bx_0|^{n-1+\mu}} \, \|f\|_{ \dot H^{1/2}(\R^n;\C^m)}, \qquad |\bx-\bx_0| \ge 3r\, .$$
\end{lem}

\begin{proof} 
By Remark \ref{rem:poincarereg}, one can change $\left({}\int_{\Omega_{+}} |u|^2\right)^{1/2}$ by 
$r\left({}\int_{\Omega_{+}} |\nabla u|^2\right)^{1/2}$ in the right hand side of the estimate of Lemma \ref{lemma1}. The latter  is controlled by $r\|f\|_{ \dot H^{1/2}(\R^n;\C^m)}$ by the existence theory for energy solutions. 
 \end{proof}

Here $3r$ is for convenience of the statements and can be changed  to $(1+\varepsilon)r$ for any $\varepsilon>0$. 

It is worth relating the above results to solutions constructed by harmonic measure, even if we do not use this estimate. 

\begin{lem}\label{lem:harmmeasure} Assume $1+n\ge 2$ and  $A(\bx)$ has scalar and  real (not necessarily $t$-independent) coefficients  and ellipticity is taken in the usual pointwise sense. Then for all Lispchitz functions $f$ with bounded support in a surface ball $\Delta(x_{0},r)$, the solution $u$ with boundary data $f$ given by harmonic measure for $\divv A \nabla$ is an energy solution.  Hence it agrees with the solution given in Lemma \ref{lemma3}. In particular, it has further the estimate for any $\rho>r$
\begin{equation}
\label{eq:sizeharmonic }
|u(\bx)|\leq C   
\frac{r^{n+\mu}}{|\bx-\bx_0|^{n-1+\mu}} \, \|\nabla f\|_{\infty},\qquad |\bx-\bx_0| \ge \rho,
\end{equation}
so that $u \to 0$ at infinity.
\end{lem}

\begin{proof}  First notice that writing $f=f_{+}-f_{-}$, the positive and negative parts both satisfy  the same assumptions as $f$. Hence we may assume $f\ge 0$.

Let $R>2r$ and  $\Omega_{R}=\reu\cap B({\bx_{0}}, R)$. Now let us recall the construction of the solution given by harmonic measure on $\reu$ taken from granted the construction on bounded domains (See \cite{Ke}).    Let $\omega_{R}^\bx$ be the harmonic measure for $\divv A  \nabla $ on $\Omega_{R}$ at pole $\bx$.  Hence $\bx\mapsto u_{R}(\bx)=\int_{\partial\Omega_{R}} f \, d\omega_{R}^\bx $ is the unique continuous function on $\overline{\Omega_{R}}$, solution  of the classical Dirichlet problem  $\divv A \nabla u_{R}=0$ with $u|_{\partial \Omega_{R}}=f$, where we have naturally extended $f$ by 0 on $\pd\Omega_{R}\cap \reu$.  It is also an energy solution on $\Omega_{R}$, 
${}\int_{\Omega_{R}}|\nabla u_{R}(\bx) |^2\, d\bx  $  bounded by a uniform constant. Indeed, it is constructed as $u_{R}=\phi_{R}+F$ where $F$ is a fixed Lipschitz extension of $f$ and $\phi_{R}$ solves $\divv A\nabla \phi_{R} = -\divv A \nabla F$ with $\phi_{R}\in W^{1,2}_{0}(\Omega_{R})$, so that the constant in the energy inequality depend on the Lipschitz norm of $f$ and the ellipticity constants of $A$.   

 Using the maximum principle of Stampacchia and the positivity of $f$,  we have
$0\le u_{R}\le u_{R'}\le \sup_{\R^n }f$ in $\overline{\Omega_{R}}$ when $R<R'$.   Thus for any $\bx\in \overline\reu$,  $u_{R}(\bx)$  converges to a finite number $u(\bx)$ as $R\to \infty$ (with $u(0,\cdot)=f$ on $\R^n$ since $u_{R}(0,\cdot)=f$ on $\Delta(x_{0},R)$ ). Already, this and the  density of  the space of compactly supported Lipschitz continuous functions on $\R^n$ into the space of  compactly supported continuous functions  imply that $\omega_{R}^\bx|_{\R^n}$ converges weakly to a finite positive measure on $\R^n$, denoted  $\omega^\bx$, and that $u(\bx)= \int_{\R^n} f\, d\omega^\bx$. 
Also, by Harnack's principle and Ascoli's theorem, $u_{R}$ (naturally extended by 0 outside $\overline{\Omega_{R}}$) converges locally uniformly to $u$ on $\reu$.  Next, this extension of $ u_{R}$ is an element of $\dot H^{1}(\reu)$, form a bounded family in that  space. It easily follows that $u$  is an energy solution of $\divv A \nabla u=0$ in $\reu$ by a weak limit argument 
 with $u|_{t=0}=f$. By uniqueness in Lemma \ref{lemma2}, $u$ is the only one. The rest of the proof is left to the reader. 
\end{proof}

We turn to decay estimates useful for Neumann solutions.

\begin{lem}\label{lem:sizeNeumann}    Let $u\in \mE$ be an  energy solution of $\divv A \nabla u=0$ 
whose conormal derivative at the boundary is further integrable and supported in some boundary ball $\Delta(x_{0},r)$. After a suitable choice of the constant of integration, we have
$$
|u(\bx)| \le C\frac{r^\mu }{{|\bx-\bx_{0}|}^{n-1+\mu}} \|\pd_{\nu_{A}}u|_{t=0}\|_{1}
$$
whenever $|\bx-\bx_{0}|\ge 2r$ for some $C$ depending on the assumptions on $A$, with $\bx_{0}=(0,x_{0})$. In particular, $u\to 0$ at infinity in any direction.
\end{lem}

\begin{proof}
Let $\alpha=\pd_{\nu_{A}}u|_{t=0}$.  We assumed $u$ belongs to the energy class, so it is determined up to a constant. We shall select one in a moment. Using the reflection principle, we see that the even extension of $u$ across the boundary is a solution of the equation 
 $${}\int_{\R^{1+n}} A^\sharp \nabla u^\sharp \cdot \overline{\nabla \phi}\, d\bx=- 2\int_{\R^{n}} \alpha(x)    \overline{\phi(0,x)}\, dx = -2 \langle \alpha \delta |_{t=0}, \phi \rangle$$ for all $\phi \in C^1_{0}(\R^{1+n}; \C^m)$ where $A^\sharp $ is the reflected matrix of $A$ (the $-$ sign in the formula comes from our convention for $\partial_{\nu_{A}}$).  Observe that the bracket is the $\dot H^{-1}(\ree; \C^m)$, $\dot H^1(\ree; \C^m)$ duality by seeing  $\alpha \delta |_{t=0}  \in \dot H^{-1}(\ree; \C^m)$ from trace theory.  Let $L^\sharp=\divv A^\sharp \nabla$ on $\R^{1+n}$.  By invertibility of $L^\sharp$  we have $u^\sharp=-2 {L^\sharp}^{-1}(\alpha \delta |_{t=0})$ in $\dot H^{1}(\ree; \C^m) $,   which means that the two agree up to a constant. Under the assumption of the lemma,  Proposition \ref{prop:fund} applies to $A^\sharp$  and let $\Gamma^{A^\sharp} $ be the fundamental  solution of $ \divv A^\sharp \nabla$. 
Using the fact that  $\alpha \in L^1$ with support in the surface ball $\Delta(x_{0},r)$, we have up to a constant for $|\bx-\bx_{0}|\ge 2r$,
$$
2{L^\sharp}^{-1}(\alpha \delta |_{t=0})(\bx)= 2\int_{\R^n} \Gamma^{A^\sharp}(\bx; 0,y) \alpha(y)\, dy   
$$
as the integral converges from the size condition \eqref{eq:gammasize}. We now choose the constant of integration so as $u^\sharp(\bx)$ agrees with this integral when $|\bx-\bx_{0}|\ge 2r$. As $\alpha$ is the conormal derivative of $u$ and is integrable, we  have necessarily $\int \alpha=0$. Thus
$$
u^\sharp(\bx)=-2 \int_{\R^n} (\Gamma^{A^\sharp}(\bx; 0,y) - \Gamma^{A^\sharp}(\bx; 0,x_{0}))\alpha(y)\, dy.
$$
 Then \eqref{eq:gammaholder} readily gives the desired estimate using, in addition, the support of $\alpha$.
\end{proof}

\section{Short review of the first order formalism}\label{sec:firstorder}

In this section, we assume that the matrix $A(\bx)$ is bounded, measurable, $t$-independent (\textit{i.e.}, $A(\bx)=A(x)$ when $\bx = (t,x)$)  and satisfies the accretivity assumption \eqref{eq:accrassumption} on $\R^n$.  It is convenient to write $A$ in a  $2 \times 2$ block form. Identifying $\C^{(1+n)m}=(\C^m)^{1+n}=\C^m \times (\C^m)^n$, $A(x)$ takes the form of a $2\times 2$ matrix
$$
  A(x)= \begin{bmatrix} a(x) & b(x)\\ c(x) &d(x) \end{bmatrix},
$$
where $a(x)\in \mL(\C^m)$, etc...
Call $\mA$ the set of such $2 \times 2$ block matrices $A$. 

Following \cite{AAMc} and \cite{AA}, one can characterize weak solutions  $u$ to the divergence form equation (\ref{eq:divform}), 
by replacing $u$ by its conormal gradient $\nabla_{A}u$ as the unknown function.
More precisely (\ref{eq:divform}) for $u$ is replaced by  \eqref{eq:firstorderODE} for 
$$
  F(t,x) =\nabla_A u(t,x)=  \begin{bmatrix} \pd_{\nu_A}u(t,x)\\ \nabla_x u(t,x) \end{bmatrix},$$ 
and
$\pd_{\nu_A}u(t,x):= (A\nabla_{t,x}u)_\no$ denotes the upward conormal derivative of $u$, that is the first component of $A\nabla_{t,x}u$, consistently with earlier notation. 
Here we use the notation $v= \begin{bmatrix} v_{\no}\\ v_{\ta}\end{bmatrix}$ for vectors  in $(\C^m)^{1+n}$ and $v_{\no}\in \C^m$ is called the scalar part and $v_{\ta}\in (\C^m)^n$ the tangential part of $v$. For example, $\pd_{t}u=(\nabla_{t,x }u)_{\no}$ and $\nabla_{x} u = (\nabla_{t,x }u)_{\ta}$.

We remark that there is the pointwise comparison $|\nabla u |\sim |\nabla_{A}u|$.

\begin{prop}  \label{prop:divformasODE}
  The pointwise transformation 
   \begin{equation}
\label{eq:hat}
A\mapsto \hat A:=  \begin{bmatrix} 1 & 0  \\ 
    c & d \end{bmatrix}\begin{bmatrix} a & b  \\ 
    0 & 1 \end{bmatrix}^{-1}=  \begin{bmatrix} a^{-1} & -a^{-1} b  \\ 
    ca^{-1} & d-ca^{-1}b \end{bmatrix}
\end{equation}
is a self-inverse bijective transformation of the set of  matrices in $\mA$.

For a pair of coefficient matrices $A= \hat B$ and $B= \hat A$, 
the pointwise map $\nabla_{t,x}u\mapsto F= \nabla_{A}u$ gives 
a one-one correspondence, with inverse $F \mapsto \nabla_{t,x}u= \begin{bmatrix} (B F)_\no \\ F_\ta  \end{bmatrix}$,
between   gradients of weak solutions $u\in  H^1_{\loc}(\R^{1+n}_{+};\C^m)$ to (\ref{eq:divform})
and   solutions $F\in L^2_{\loc}(\R^{1+n}_{+}; (\C^m)^{1+n})$  of the generalized Cauchy--Riemann equations
\begin{equation}  \label{eq:firstorderODE}
  \pd_t F+  \begin{bmatrix} 0 & \divv_{x} \\ 
     -\nabla_{x} & 0 \end{bmatrix} B F=0  \quad \mathrm{and} \quad \curl_{x}F_{\ta}=0, 
\end{equation}
where the derivatives are taken in the $\R^{1+n}_+$ distributional sense.\end{prop}

This originates from \cite{AAMc} and is proved in this generality in \cite{AA}.   
Denote by  $\dirac $  the self-adjoint operator on $\mH= L^2(\R^n; (\C^m)^{1+n})$ defined by $$
  \dirac:= 
    \begin{bmatrix} 0 & \divv_{x} \\ 
     -\nabla_{x} & 0 \end{bmatrix}\quad \mathrm{with}\quad \dom(\dirac) = \begin{bmatrix}  \dom(\nabla)\\ \dom (\divv)\end{bmatrix}. 
$$
 The closure of the range of $\dirac$ is the set of $F\in \mH$ such that $\curl_{x}F_{\ta}=0$, that is $\clos{\ran(\dirac)}=\mH^0$.  It is shown in \cite{AKMc} that the operators $\dirac B$ and  $B\dirac $ with respective domains $B^{-1}\dom(\dirac)$  and $\dom(\dirac)$  are bisectorial operators with bounded holomorphic functional calculi on the closure of their range $\mH^0$ and $B\mH^0$ respectively. Observe the similarity relation
 \begin{equation}
\label{eq:similar}
B(\dirac B)= (B\dirac) B \quad \mathrm{on}\ \dom(\dirac B)
\end{equation}
that allows to transfer functional properties between $\dirac B$ and $B\dirac$. 
In particular, if $\sgn(z)=1$ for $\re z>0$ and $-1$ for $\re z<0$, the operators $\sgn(\dirac B)$ and $\sgn(B\dirac)$ are well-defined bounded involutions on $\mH^0$ and $B\mH^0$ respectively. 
 One defines the spectral spaces $\mH^{0, \pm}_{\dirac B}= \nul(\sgn(\dirac B)\mp I)$ and 
 $\mH^{0, \pm}_{B\dirac }= \nul(\sgn(B\dirac )\mp I)$. They topologically split $\mH^0$ and $B\mH^0$ respectively. The restriction of $\dirac B$  to the invariant space $ \mH^{0, +}_{\dirac B}$ is sectorial of type less than $\pi/2$, hence it generates an analytic semi-group $e^{-t\dirac B}$, $t\ge 0$, on it. Similarly, the restriction of $B\dirac $  to the invariant space $ \mH^{0, +}_{B\dirac }$ is sectorial of type less than $\pi/2$, hence it generates an  analytic semi-group $e^{-tB\dirac }$, $t\ge 0$, on $ \mH^{0, +}_{B\dirac }$.

 \begin{thm} \label{thm:AAM-AA1} Let $u\in H^1_{loc}(\R^{1+n}_{+};\C^m)$. The function
 $u$ is  a weak solution of $\divv A \nabla u=0$ with 
  $ \|\tN(\nabla u)\|_{2}<\infty$  if and only if there exists $F_{0} \in \mH^{0,+}_{\dirac B}$ such that  
$\nabla_{A}u=e^{-t \dirac B} F_{0}$.  Moreover, $F_{0}$ is unique  and $\|F_{0}\|_{2}\approx 
 \|\tN(\nabla u)\|_{2}$. We set $\nabla_{A}u|_{t=0}:=F_{0}$.
 \end{thm}

The if part was obtained in \cite{AAMc} and the only if part in \cite[Theorems 8.2]{AA}.
Here $ \tN(g)$ is the  Kenig-Pipher modified non-tangential function
where  $$
  \tN(g)(x):= \sup_{t>0}  t^{-(1+n)/2} \|g\|_{L_2(W(t,x))}, \qquad x\in \R^n,
$$
with $W(t,x):= (c_0^{-1}t,c_0t)\times \Delta(x,c_1t)$, for some fixed constants $c_0>1$, $c_1>0$. 
A remark is that the same proof shows when coefficients are $t$-independent  that for the equivalence to hold one could replace $ \|\tN(\nabla u)\|_{2}$ by 
$\sup_{t>0} (\frac 1 t \int_{t}^{2t}\|\nabla_{t,x}u\|_{2}^2\,  ds)^{1/2}$ or  the stronger $\sup_{t>0 } \|\nabla_{t,x}u\|_{2}$ or even the square function $({}\int_{\R^{1+n}_{+}}t|\pd_{t}\nabla_{t,x}u|^2 dt dx)^{1/2} $, so that in the end all these quantities are \textit{a priori} equivalent for weak solutions.

Let us pursue further the discussion by extending this to Sobolev spaces with negative order.   Say that $u\in \mE_{s}$ with $s<0$ if  ${}\int_{\R^{1+n}_{+}}t^{-2s-1}|\nabla_{t,x}u|^2 dt dx<\infty$ while  $u\in \mE_{0}$ if  
$\|\tN (\nabla u)\|_{2}<\infty$.  With this notation $\mE_{-1/2}=\mE$.
\begin{prop}\label{prop:ext} Let  $s\in [-1,0)$. 
\begin{enumerate}
  \item The operator $\dirac B|_{\clos{\ran({\dirac})}}$ can be extended to a bi-sectorial operator on the homogeneous Sobolev space $\dot\mH^{s}$ which is the closure of $\clos{\ran({\dirac}})= \dot\mH^0$ for the homogeneous Sobolev norm $\|(-\Delta)^{-s/2}f\|_{2}$. This operator, which we keep writing $\dirac B$ for simplicity, has a bounded holomorphic functional calculus on $\dot\mH^{s}$. In particular, the operators $\sgn(\dirac B)\mp I$ are well-defined projections on  $\dot\mH^{s}$ and their ranges $\dot\mH^{s, \pm}_{\dirac B}$ form a splitting of $\dot\mH^{s}$.  
  \item Let $u\in H^1_{loc}(\R^{1+n}_{+})$. 
Then $u$ is a weak solution of $\divv A \nabla u=0$ in $\reu$ with $u\in \mE_{s}$  if and only if there exists $ F_{0} \in \dot \mH^{s, +}_{\dirac B}$  such that  $\nabla_{A}u=e^{-t\dirac B}   F_{0}$. Moreover, $F_{0}$ is unique and $$\|F_{0}\|_{\dot\mH^{s}}\approx \left({}\int_{\R^{1+n}_{+}}t^{-2s-1}|\nabla_{t,x}u|^2 dt dx\right)^{1/2}.$$ We set $\nabla_{A}u|_{t=0}:=F_{0}$.
\end{enumerate}
\end{prop}

Here $ \Delta$ is the self-adjoint Laplacian acting componentwise on $L^2(\R^n; (\C^m)^{1+n})$. It agrees with $-D^2$ on $\clos{\ran({\dirac}})$. 

\begin{proof}  
Item (1) is in Proposition 4.5 of \cite{AMcM} where $\dirac B|_{\clos{\ran({\dirac})}}$ is called $\uT$ there. Item 2 for $s=-1$ is Corollary 4.5 of  \cite{AMcM}, for $s=-1/2$ is Proposition 4.7 of \cite{AMcM}. The other cases are treated in \cite{R2}. 
\end{proof}

\begin{rem}
We have introduced a notion of conormal gradient at the boundary $\nabla_{A}u|_{t=0}$ for solutions in $\mE_{s}$. Strictly speaking this notion depends on $s$ as well and in particular for $s=-1/2$, we recover the notions already defined for energy solutions.  What allows us not to distinguish $s$ in the notation is that it is a consistent notion for two different values of $s$. More precisely, if $u\in \mE_{s}\cap \mE_{s'}$ with $s,s'\in [-1,0]$, then the convergence of $\nabla_{A}u(t,\cdot)$ as $t\to 0$ is both in $\dot \mH^s$ and $\dot\mH^{s'}$, hence the limits agree in the space of distributions. 
\end{rem}

\section{Boundary layer operators}

In this section, we assume that $A$ is  bounded measurable $t$-independent matrix  which is strictly accretive on $\mH^0$, that is satisfying \eqref{eq:accrassumption}.

It has been proved recently in \cite{R1} using the functional calculi for $\dirac B$ and  $B\dirac$    that  the classical single and double layer operators for $\divv A \nabla$, $\nabla S_{t}$ and $D_{t}$, can be defined    as $L^2$  bounded operators, uniformly with respect to $t>0$, with limits at $t=0$. 
More precisely,  for $t>0$, define $\nabla_{A}\mS_{t}$ and $\mD_{t}$ for $h\in L^2(\R^n;\C^m)$ by
\begin{equation}
\label{eq:gradst}
(\nabla_{A}\mS_{t}h)(x):=\bigg(e^{-t\dirac B}X_{+}(\dirac B) \begin{bmatrix} h \\ 0\end{bmatrix}\bigg)(x)
\end{equation}
and 
\begin{equation}
\label{eq:dt}
(\mD_{t}h)(x):=- \bigg(e^{-tB\dirac} X_{+}(B\dirac) \begin{bmatrix} h \\ 0\end{bmatrix}\bigg)_{\no}(x),
\end{equation}
where $X_{+}(z)=1$ if $\re z>0$ and 0 if $\re z <0$ so that $X_{+}(z)=\frac 1 2( \sgn(z)+1).$ Remark that at this general level, there is an abuse of language as the operator $\mS_{t}$ is not defined (although it will in $\dot H^1(\R^n;\C^m)$), only $\nabla_{A}\mS_{t}$ is.  It follows from the bounded holomorphic functional calculus for $\dirac B$ and $B\dirac$ that the right hand sides  are $L^2$-bounded operators and have strong  limits when $t\to 0$. 

\begin{lem}
 Whenever $h \in W^{1,2}(\R^n;\C^m)$,
\begin{equation}
\label{eq:graddt}
\nabla_{A}\mD_{t}h= -e^{-t\dirac B} X_{+}(\dirac B) \begin{bmatrix} 0 \\ \nabla h \end{bmatrix}.
\end{equation}
\end{lem}

\begin{proof}
From the calculations in   \cite{AA}, we have
\begin{align*}
\label{}
\nabla_{A}\bigg(e^{-tB\dirac} X_{+}(B\dirac) \begin{bmatrix} h \\ 0\end{bmatrix}\bigg)_{\no}    & 
= -\dirac e^{-tB\dirac} X_{+}(B\dirac) \begin{bmatrix} h \\ 0\end{bmatrix}  \\
    &   = -\dirac e^{-tB\dirac} X_{+}(B\dirac) \begin{bmatrix} h \\ 0\end{bmatrix}
    \\
    &
    = -e^{-t\dirac B} X_{+}( \dirac B)\dirac\begin{bmatrix} h \\ 0\end{bmatrix}
    \\
    & = + e^{-t\dirac B} X_{+}(\dirac B) \begin{bmatrix} 0 \\ \nabla h \end{bmatrix}.
    \end{align*}
\end{proof}

The right-hand side in \eqref{eq:graddt} makes sense for any distribution $h$ such that $\nabla h\in L^2(\R^n; (\C^m)^n),$  that is, $h\in \dot W^{1,2}(\R^n;\C^m)$.

\begin{lem}[Boundary layer representation] \label{lem:BLP} Assume that $u\in \mE_{0}$, i.e.,  $\|\tN(\nabla u)\|_{2}<\infty$. Then 
$$
 \nabla_{A}u(t,\cdot)=  \nabla_{A}\mS_{t}(\pd_{\nu_{A}}u|_{t=0}) -  \nabla_{A}\mD_{t}(u|_{t=0})
 $$
where $\nabla_{A}\mD_{t}(u|_{t=0})$ is interpreted as the right hand side of  \eqref{eq:graddt}.
The equality holds in $\mE_{0}\cap C([0,+\infty); \dot\mH^{0, +}_{\dirac B})$.
\end{lem}  

\begin{proof}
Using Theorem \ref{thm:AAM-AA1}, if $\|\tN(\nabla u)\|_{2}<\infty$ then 
$\nabla_{A}u=e^{-t \dirac B} F_{0}$, with  $F_{0} \in \mH^{0,+}_{\dirac B}$ and $F_{0}=\nabla_{A}u|_{t=0}$. As $X_{+}(\dirac B)$ is a projection on $ \mH^{0,+}_{\dirac B}$, we have $X_{+}(\dirac B)F=F$  when $F\in \mH^{0,+}_{\dirac B}$, so that
\begin{align*}
\label{}
 \nabla_{A}u   & = e^{-t \dirac B} \begin{bmatrix} \pd_{\nu_{A}}u|_{t=0}\\ \nabla_{x}u|_{t=0}\end{bmatrix}   \\
    &  =
e^{-t \dirac B}X_{+}(\dirac B) \begin{bmatrix} \pd_{\nu_{A}}u|_{t=0}\\ \nabla_{x}u|_{t=0}\end{bmatrix}
\\
&
= 
e^{-t \dirac B}X_{+}(\dirac B) \begin{bmatrix} \pd_{\nu_{A}}u|_{t=0}\\ 0\end{bmatrix} + e^{-t \dirac B}X_{+}(\dirac B) \begin{bmatrix} 0\\ \nabla_{x}u|_{t=0}
\end{bmatrix}
\\
&
 =\nabla_{A}\mS_{t}(\pd_{\nu_{A}}u|_{t=0}) -  \nabla_{A}\mD_{t}(u|_{t=0}).
\end{align*}
\end{proof}

\begin{rem} If one can make sense of both $S_{t}(\pd_{\nu_{A}}u|_{t=0})$ and $\mD_{t}(u|_{t=0})$ as distributions and fixing the constants of integration, one has the representation
$$
u=\mS_{t}(\pd_{\nu_{A}}u|_{t=0}) -  \mD_{t}(u|_{t=0}).\footnote{Since we have upward convention for conormal derivatives and fundamental solutions for $\divv A \nabla$ (usually it is taken for $-\divv A \nabla$), we obtain the same sign rule as in the usual Green's formula, due to the cancellation of two minus signs. Had we been working in the lower-half space though,  the upward normal is the outward normal and the sign rule would be opposite.}
$$
This,  of course,  is the classical formula obtained from Green's theorem if one can write 
$\mS_{t}$ and $\mD_{t}$ in integral form using the fundamental solution of $L^*$.  We come to this below. 
\end{rem}

\begin{cor}[Generalized boundary layer representation] \label{lem:GBLP} Let $s\in [-1,0]$, and  $u \in \mE_{s}$ be a weak solution of $Lu=0$ in $\reu$.  Then 
$$
 \nabla_{A}u(t,\cdot)=  \nabla_{A}\mS_{t}(\pd_{\nu_{A}}u|_{t=0}) -  \nabla_{A}\mD_{t}(u|_{t=0})
 $$
where $\nabla_{A}\mD_{t}(u|_{t=0})$ is interpreted as the right hand side of  \eqref{eq:graddt} with $\nabla h \in \dot H^{s}_{\nabla} (\R^n; (\C^{m})^n)$ and 
$\nabla_{A}u|_{t=0}=F_{0} $ given by Theorem \ref{thm:AAM-AA1} for $s=0$ and Proposition \ref{prop:ext} for $s<0$.  The equality holds in  $\mE_{s} \cap C([0,+\infty); \dot\mH^{s, +}_{\dirac B})$.  \end{cor} 

\begin{proof} For $s=0$, this is Lemma \ref{lem:BLP}. For $s<0$, using the extension of functional calculus of $\dirac B|_{\clos{\ran({\dirac})}}$  on $\dot\mH^s$ in Proposition \ref{prop:ext}, one defines $\nabla_{A}\mS_{t}$ on the scalar Sobolev space $\dot H^{s}(\R^n;\C^m)$ and $\nabla_{A}\mD_{t}$ by \eqref{eq:graddt} with $\nabla h \in \dot H^{s}_{\nabla} (\R^n; (\C^{m})^n)$. The proof is now the same as for $s=0$. 
\end{proof}

\begin{rem}
Let $Lu=0$ with $u\in \mE_{s}$ and $s<0$.  We know that  the semigroup equation $\nabla_{A}u(t,\cdot)=e^{-t\dirac B}\, \nabla_{A}u(0,\cdot)$ holds  in $C([0,+\infty); \dot\mH^{s, +}_{\dirac B})$. Thus for all $ \varepsilon>0$ 	and $t>0$,  $\nabla_{A}u(t+ \varepsilon,\cdot)=e^{-t\dirac B} \nabla_{A}u(\varepsilon,\cdot)$.  From ${}\int_{\R^{1+n}_{+}}t^{-2s-1}|\nabla_{t,x}u|^2 dt dx<\infty$, for almost every $\varepsilon>0$, $\nabla_{A}u(\varepsilon,\cdot) \in L^2(\R^n; (\C^m)^{1+n})$, hence  to $\nabla_{A}u(\varepsilon,\cdot) \in \clos{\ran({\dirac})}$ as any $L^2$-conormal gradient. Set $u_{\varepsilon}(t,x)= u(t+\varepsilon,x)$. By Theorem \ref{thm:AAM-AA1}, the semigroup equation implies that  $u_{\varepsilon}\in \mE_{0}$, that is $\|\tN(\nabla_{A}u_{\varepsilon})\|_{2}<\infty$  and $\nabla_{A}u_{\varepsilon}\in C ([0,\infty); \dot\mH^{0, +}_{\dirac B})$. An easy argument shows that this must hold for all $\varepsilon>0$. In particular, the generalized boundary layer representation in the statement above holds in $C ((0,\infty); L^2(\R^n, (\C^m)^{1+n})$ (not at the boundary $t=0$) as well, and even in $C^\infty ((0,\infty); L^2(\R^n, (\C^m)^{1+n})$ by semigroup theory. Thus there is instantaneous regularisation of solutions in the upper half-space.   
\end{rem}
 
 \begin{prop}[\cite{R1}]\label{prop:rosen}  Let $A$ be as in the beginning of this section. Assume further that  $\divv A\nabla $ and $\divv A^*\nabla $ satisfy the De Giorgi condition  or equivalently, the Nash local regularity condition. Let $\Gamma^A$ and $\Gamma^{A^*}$ be the fundamental solutions constructed in Proposition \ref{prop:fund}.
\begin{enumerate}
  \item[i)] For $t\in \R$, $t\ne 0$, let  $D_{t}$ be the operator given by the double layer integral (when it converges for suitable $h$)
  \begin{align}
\label{eq:DLint} D_{t}h(x)&= \int_{\R^n} \langle h(y), \pd_{\nu_{A^*}}\Gamma^{A^*}( s,y; t,x)_{|s=0}\rangle \, dy \\
\nonumber &=  \int_{\R^n}  \langle h(y), (A^*(y)\nabla_{s,y}\Gamma^{A^*}(0,y; t,x))_{\no}\rangle \, dy, \ t>0, x\in \R^n.
 \end{align} 
 Here, $\langle \  , \ \rangle$ stands for the canonical complex inner product on $\C^m$.
 Then, the abstract operator $\mD_{t}$ agrees with the usual double layer potential in the sense that one has for $h\in L^2(\R^n;\C^m)$ with compact support and $t>0$, $\mD_{t}h=D_{t}h$ and thus showing that $D_{t}$ extends to a  bounded map on $L^2(\R^n;\C^m)$, uniformly in $t>0$, with strong limit as $t\to 0$. 
  
  \item[ii)] For $t\in \R$, $t\ne 0$, let $S_{t}$ be the operator defined by the single layer integral (when it converges for suitable $h$)
   \begin{equation}
\label{eq:SLint}
   S_{t}h(x)=\int_{\R^n}\Gamma^A(t,x; 0,y) h(y)\, dy.
\end{equation} 
   Then for $h\in L^2(\R^n;\C^m)$ with compact support and $t>0$, $\nabla_{A}\mS_{t}h=\nabla_{A}S_{t}h$, 
 thus allowing to define $\mS_{t}h(x)$ by the single layer integral  $S_{t}h(x)$ and showing that $S_{t}$  extends to a bounded map, uniformly in $t$, from $L^2(\R^n;\C^m)$ into $\dot W^{1,2}(\R^n;\C^m)$, with strong limit as $t\to 0$. 
\end{enumerate}  \end{prop}

We next recall estimates on the layer potentials.  

\begin{lem} Let $A$ be as in the beginning of this section and assume that  $\divv A\nabla $ and $\divv A^*\nabla $ satisfy the De Giorgi condition  or equivalently, the Nash local regularity condition. Then
\begin{enumerate}
  \item The single layer operator $S_{t}$ maps $L^p(\R^n; \C^m)$ to $\dot W^{1,p}(\R^n; \C^m)$ for $1<p\le 2$ uniformly in $t>0$,  and converges when $t\to 0$  for the weak operator topology. 
  \item The single layer operator $S_{t}$ maps $\dot W^{-1,p}(\R^n; \C^m)$ to $L^p(\R^n; \C^m)$ for $2\le p<\infty$ uniformly in $t>0$,  and converges  when $t\to 0$  for the weak operator topology. 
  \item The double layer operator is bounded on $L^p(\R^n; \C^m)$ for $2\le p<\infty$, uniformly in $t>0$,  and converges when $t\to 0$  for the weak operator topology. 
\end{enumerate}

\end{lem}

\begin{proof} The proof of (1) and (3) is given  for equations and $1+n\ge 3$ in \cite{HMiMo}, but the arguments using the De Giorgi conditions are applicable here. We skip details.   
(2) is the dual statement of (1) as  the adjoint of  the single layer for $A$  is the single layer for $A^*$ and we use De Giorgi condition  for both. 
\end{proof}

The next result was observed in a special case as part of the proof of Theorem 5.35 in \cite{HKMP2}. It receives a much simpler proof here. 

\begin{cor}\label{cor:BL} Let $A$ be as in the beginning of this section and assume that  $\divv A\nabla $ and $\divv A^*\nabla $ satisfy the De Giorgi condition  or equivalently, the Nash local regularity condition. Let $u$ be an energy solution to $\divv A \nabla u=0$ in $\reu$. Assume that for some $2\le p <\infty$,  $u|_{t=0} \in L^p(\R^n; \C^m)$ and $\pd_{\nu_{A}}u|_{t=0}\in \dot W^{-1,p}(\R^n; \C^m)$. Then the abstract boundary layer representation
$$
u(t,x)=\mS_{t}(\pd_{\nu_{A}}u|_{t=0})(x) -  \mD_{t}(u|_{t=0})(x)
$$
holds for all $t\ge 0$ in $L^2_{loc}(\R^n;\C^m)$.  In particular, $\sup_{t\ge 0}\|u(t,\cdot)\|_{L^p(\R^n;\C^m)}<\infty$.
\end{cor}

\begin{proof} By Corollary \ref{lem:GBLP},  the equality holds   up to a constant, that is
$$
u(t,x)=\mS_{t}(\pd_{\nu_{A}}u|_{t=0})(x) -  \mD_{t}(u|_{t=0})(x) + c, \quad t>0, 
$$ 
  in $L^2_{\loc}(\reu; \C^m)$,  but also in $ L^2_{loc}(\R^n; \C^m)$ for each $t>0$ as $p\ge 2$ and the right hand side belongs to $L^p(\R^n; \C^m) + \C^m$  by the previous lemma and the left hand side is in $L^2_{\loc}(\R^n;\C^m)$. One can pass to the limit in $t\to 0$, after testing against a $C^\infty_{0}(\R^n;\C^m)$ function. For the right hand side, we use the previous lemma and for the left hand side, this is because $t\to u(t,\cdot)$ is continuous at $0$ in $L^2_{\loc}(\R^n;\C^m)$ as $u$ is an energy solution. One  obtains $u|_{t=0}(x)=\mS_{0}(\pd_{\nu_{A}}u|_{t=0})(x) -  \mD_{0^+}(u|_{t=0})(x) + c$.  As all the functions belong to $L^p(\R^n;\C^m)$, we conclude that $c=0$. 
\end{proof}

\begin{rem} The same statement holds for solutions in the classes $\mE_{s}$ for all $s\in [-1,0]$.  For $s>-1$,  $\mE_{s}$ can be shown to imbed into $C([0,\infty); L^2_{loc}(\R^n;\C^m))$, so the proof is the same.  For $s=-1$, it follows from  \cite{AA} that any solution of $\divv A \nabla u=0$ in the class $\mE_{-1}$, that is with the square function bound $\iint t|\nabla u|^2\, dtdx <\infty$,  belongs in fact to $ C([0,\infty); L^2(\R^n;\C^m))+\C^m \subset C([0,\infty); L^2_{loc}(\R^n;\C^m))$.  This is enough to finish the argument. 
 It can be shown that the boundary layer representation also holds in the space of continuous functions valued in $L^p(\R^n;\C^m)$ equipped with the weak topology. 
\end{rem}

\section{Interior non-tangential maximal estimates}

We prove here the following \textit{a priori} inequality.\footnote{This inequality will be proved in larger generality in \cite{AS} with a different argument.} See the introduction for history and differences in approach for this result. 

\begin{thm}\label{thm:main} Let $\divv A \nabla $ be a uniformly  elliptic system with $A(x)$ measurable, bounded, $t$-independent, complex coefficients on $\R^{1+n}$ with the strict G\aa rding inequality     on $\mH^0$, namely \eqref{eq:accrassumption}.
Assume that $\divv A \nabla $ and $\divv A^* \nabla $ satisfy the De Giorgi condition and call $0<\mu_{DG}$ the exponent that works for both. Then   for   all $\frac{n}{n+\mu_{DG}}<p\le 2$ and for  any weak solution of $Lu=0$ on the upper half-space $\R^{1+n}_{+}, 1+n\ge 2$, in any of the classes $\mE_{s}$, $-1\le s\le 0$, we have
\begin{equation}\label{eq:main}
\|\tN (\nabla u)\|_{p}\lesssim \|\pd_{\nu_A}u|_{t=0}\|_{H^p(\R^n; \C^m)}+\|\nabla_{x}u|_{t=0}\|_{H^p(\R^n; (\C^m)^n)},
\end{equation}
where $H^p(\R^n)$  denotes the real Hardy space if $p\le 1$ and $L^p(\R^n)$ for $p>1$.  \end{thm}

Recall that for $h\in L^2(\R^n;\C^m)$ and $t>0$, 
\begin{equation*}
\nabla_{A}\mS_{t}h=e^{-t\dirac B}X_{+}(\dirac B) \begin{bmatrix} h \\ 0\end{bmatrix}
\end{equation*}
and  for $h\in W^{1,2}(\R^n;\C^m)$ and $t>0$, 
\begin{equation*}
\nabla_{A} \mD_{t}h=- e^{-t\dirac B} X_{+}(\dirac B) \begin{bmatrix} 0 \\ \nabla h\end{bmatrix}.
\end{equation*}
Remark that $\nabla$ means here the tangential gradient $\nabla_{x}$, while $\nabla_{A}$ still means the full conormal gradient. Recall also that the size of the full conormal gradient is pointwise comparable to that of the full gradient  $\nabla_{t,x}$.

It is convenient to set $H^p_{\nabla}(\R^n;(\C^m)^n)=\nabla \dot W^{1,p}(\R^n;\C^m)$ (or again, those $L^p$ curl-free functions) for $p>1$ (for $p\le1$ it was defined in Section \ref{sec:rellichp<1}) and to define the operator $V_t$ on $H^2_{\nabla} (\R^n;(\C^m)^n)$ by $$V_t{g}= -e^{-t\dirac B} X_{+}(\dirac B) \begin{bmatrix} 0 \\ g\end{bmatrix}$$ for any $g\in \dot H^{2}_{\nabla}(\R^n;(\C^m)^n)$.  

Theorem  \ref{thm:main} follows immediately from the next   \textit{a priori} boundedness result,  together with Proposition \ref{lem:GBLP}.

\begin{thm}\label{thm:main2} Let $L$ be as in Theorem \ref{thm:main}. Then    for $\frac{n}{n+\mu_{DG}}<p\le 2$,
\begin{equation}
\label{eq:st}
\|\tN (\nabla_{A} \mS_{t}h)\|_{p}\lesssim \|h\|_{H^p(\R^n;\C^m)}
\end{equation}
\begin{equation}
\label{eq:delta}
\|\tN (V_t{g})\|_{p}\lesssim \|{g}\|_{H^p_{\nabla}(\R^n;(\C^m)^n)}
\end{equation}
This means that there is a  linear extension of the map $h\mapsto (\nabla_{A} \mS_{t}h)_{t>0} $ defined on $L^2(\R^n;\C^m)\cap  H^p(\R^n;\C^m)$  to $H^p(\R^n;\C^m)$ with such an estimate,    and of the map ${g}\mapsto (V_t{g})_{t>0}$  from  $H^2_{\nabla}(\R^n;(\C^m)^n)  \cap H^p_{\nabla}(\R^n;(\C^m)^n)$ to all of $H^p_{\nabla}(\R^n;(\C^m)^n)$ with such an estimate. \end{thm}

In particular, this yields 
\begin{equation}\label{eq35}
\|\tN (\nabla_{A}\mD_{t}h)\|_{p}\lesssim \|\nabla h\|_{H^p_{\nabla}(\R^n;(\C^m)^n)}
\end{equation}
whenever $h\in L^2(\R^n;\C^m)$ and $\nabla h\in L^2(\R^n;(\C^m)^n)$ as well (in fact more general $h$ can be used provided one makes sense of the various objects).

We use the notation $N^p_{2}(\R^n)$ to denote the (quasi-)Banach space of all $L^2_{loc}(\R^{1+n}_{+})$ functions such that $\|\tN (f)\|_{p}<\infty$, $0<p<\infty$. These spaces are further studied in \cite{HR}. See also \cite{Hu} for a more systematic approach.

For this purpose, we  use again the 2-atoms for $H^p_{\nabla}(\R^n; \C^n)$ but in a slightly different way.

\begin{lem}\label{lem:hpgrad} Let $\frac{n}{n+1}<p \le 1$. 
\begin{enumerate}
  \item  $H^p_{\nabla}(\R^n;\C^n)$ has the following atomic characterization: Let $g\in \mD'(\R^n;\C^n)$. Then $g\in H^p_{\nabla}(\R^n;\C^n)$ if and only if $g=\sum \lambda_{j} a_{j}$ in $\mD'(\R^n;\C^n)$ with $\sum |\lambda_{j}|^p<\infty$ and $a_{j}$ are $2$-atoms for $H^p_{\nabla}(\R^n;\C^n)$. 
 Moreover, $\|g\|_{H^p_{\nabla}(\R^n;\C^n)}\sim \inf \|(\lambda_{j})\|_{\ell^p}$ with the infimum taken over all such decompositions.
  \item  $H^p_{\nabla}(\R^n;\C^n) \cap H^2_{\nabla}(\R^n;\C^n)$  is the subspace of $H^p_{\nabla}(\R^n;\C^n)$ of those $g$ having an atomic decomposition with $\|g\|_{H^p_{\nabla}(\R^n;\C^n)} \sim \|(\lambda_{j})\|_{\ell^p}$ and  which converges also in $H^2_{\nabla}(\R^n;\C^n)$. It is dense in $H^p_{\nabla}(\R^n;\C^n)$. 
  \item A bounded linear operator $T:  H^2_{\nabla}(\R^n;\C^n) \to N^2_{2}(\R^n)$ with $\sup \|\tN (Ta)\|_{{p}}<\infty$, where the supremum is taken over all $2$-atoms for $H^p_{\nabla}(\R^n;\C^n)$, extends to a bounded map from  $H^q_{\nabla}(\R^n;\C^n)$ to $N^q_{2}(\R^n)$, for $p\le q \le 2$. 
  \item A bounded linear operator $T:  H^r_{\nabla}(\R^n;\C^n) \to N^r_{2}(\R^n)$ for some $1<r<2$ with $\sup \|\tN (Ta)\|_{p}<\infty$, where the supremum is taken over all $2$-atoms for $H^p_{\nabla}(\R^n;\C^n)$, extends to a bounded map from  $H^q_{\nabla}(\R^n;\C^n)$ to $N^q_{2}(\R^n)$, for $p\le q \le r$.
 
\end{enumerate}
All this extends straightforwardly to $H^p_{\nabla}(\R^n;(\C^m)^n)$ spaces.
\end{lem}

\begin{proof} The proof of (1)  is done in \cite{LMc} when $p=1$. As already mentioned, the method in \cite{LMc} is to construct a Calder\'on reproducing formula  that allows to see that  $H^p_{\nabla}(\R^n)$  is a retract of  the tent space $T^p_{2}$ of \cite{CMS} : $H^p_{\nabla}(\R^n;\C^n)$ is isomorphic to  closed and complemented subspace of $T^p_{2}$ and one can use the atomic decomposition of $T^p_{2}$ which, given the particular form of the retract mappings in \cite{LMc}, gives (1). 
Their method extends  to the range $\frac{n}{n+1}<p \le 1$ without difficulty. We skip details.

The proof of (2) is as follows.  The retract mappings are of Littlewood-Paley type with smooth and compactly supported convolution kernels with mean 0 so they work simultaneously and boundedly for all $\frac{n}{n+1}<p <\infty$. Denoting by $S$ the mapping from $T^p_{2}$ to $H^p_{\nabla}(\R^n;\C^n)$ of the retract diagram, we have $S(T^p_{2}\cap T^2_{2})= H^p_{\nabla}(\R^n;\C^n) \cap H^2_{\nabla}(\R^n;\C^n)$. Thus, it suffices to show that $T^p_{2}\cap T^2_{2}$ is the subspace of $T^p_{2}$ of those elements having a $T^p_{2}$ atomic decomposition that converges also in $T^2_{2}$.  This fact is implicit in the proof of \cite[Theorem 4.9, step 3]{AMcR} for $p=1$ and the very same argument applies when $p<1$. Again we skip details.  A different and explicit method is in \cite{JY}, Proposition 3.1. The density follows from the density of $T^p_{2}\cap T^2_{2}$ in $T^p_{2}$. 

The proof of (3) is now simple using (2).  To prove the boundedness at $q=p$, 
choose an atomic decomposition $\sum \lambda_{j}   a_{j}$ for $g\in H^p_{\nabla}(\R^n;\C^n) \cap H^2_{\nabla}(\R^n;\C^n)$ that  converges also in $H^2_{\nabla}(\R^n;\C^n)$. Then this convergence and boundedness of $T$ imply $Tg= \sum \lambda_{j}Ta_{j}$ and it follows that 
$\|Tg\|_{N^p_{2}(\R^n)} \lesssim  \|g\|_{H^p_{\nabla}(\R^n;\C^n)}$ using $\sup \|\tN (Ta)\|_{p}<\infty$. It remains to extend by density.  
The boundedness when $p<q<2$ follows by interpolation. The spaces $H^p_{\nabla}(\R^n;\C^n)$ for $\frac{n}{n+1}<p<\infty$ interpolate by the retract property and the interpolation property of the tent spaces (\cite{CMS}, \cite{Be} and \cite{CV} for $p<1$). The result follows by using real interpolation for the sublinear operator $g\mapsto \tN (Tg)$.

We finish with the proof of (4). Remark that $2$-atoms for $H^p_{\nabla}(\R^n;\C^n)$ are elements of $H^r_{\nabla}(\R^n;\C^n)$ as $r<2$, so the statement is meaningful. It is enough to prove the boundedness at $q=p$ as interpolation takes care of the other values of $q$.   Choose an atomic decomposition $\sum \lambda_{j}   a_{j}$ for $g\in H^p_{\nabla}(\R^n;\C^n) \cap H^2_{\nabla}(\R^n;\C^n)$ that  converges also in $H^2_{\nabla}(\R^n;\C^n)$.  Of course, one has also the convergence in $H^p_{\nabla}(\R^n;\C^n)$. Interpolation implies that it converges also in $H^r_{\nabla}(\R^n;\C^n)$. Thus $Tg= \sum \lambda_{j}Ta_{j}$ by boundedness of $T$ at exponent $r$ and it follows that 
$\|Tg\|_{N^p_{2}(\R^n)} \lesssim  \|g\|_{H^p_{\nabla}(\R^n;\C^n)}$ using $\sup \|\tN (Ta)\|_{p}<\infty$. It remains to extend by density. \end{proof}

\begin{proof}[Proof of Theorem \ref{thm:main2}]  We prove \eqref{eq:delta}.  By lemma \ref{lem:hpgrad} it is enough to prove the bound for $2$-atoms for $H^p_{\nabla}(\R^n;(\C^m)^n)$ when $\frac{n}{n+\mu}=p>p_{0}=\frac{n}{n+\mu_{0}}$ with $0<\mu_{0}<\mu_{DG}$. Fix such a $p$.  The argument follows a method of Kenig-Pipher \cite{KP}. 
Let  $a=\nabla b$ be a $2$-atom for  $H^p_{\nabla}(\R^n;(\C^m)^n)$, with $a,b$  supported in a surface ball $\Delta(x_{0},r)$. 
We note that in this case $V_t a=  \nabla_{A} \mD_{t}b$ as both $a,b$ are $L^2$ functions. 
As  our techniques are scale invariant, we assume  that $x_{0}=0$ and $r=1$ to simplify the exposition. We let $\Delta_{k}=\Delta(0,2^k)$ and $C_{k}=\Delta_{k+1}\setminus \Delta_{k}$ for $k\in \N$.  We have
\begin{align*}
\|\tN (\nabla_{A} \mD_{t}b)\|_{L^p(\Delta_{2})} &\le |\Delta_{2}|^{1/p-1/2}  \|\tN (\nabla_{A} \mD_{t}b)\|_{L^2(\Delta_{2})}
\\ &\le C \|\nabla b\|_{2} |\Delta_{2}|^{1/p-1/2} \le C 4^{n(1/p-1/2)}.
\end{align*}

It remains to show $\|\tN (\nabla_{A} \mD_{t}b)\|_{L^2(C_{k})} \le C2^{-k(n/2 +\mu_{0})}=C2^{-k(n/p_{0}-n/2)}$  when $k\ge 2$,  which implies $\|\tN (\nabla_{A} \mD_{t}b)\|_{L^p(C_{k})}^p \le C2^{-k(n/p_{0}-n/p)p}$. Indeed, summing all these estimates for $k\ge 1$, yields $\|\tN (\nabla_{A} \mD_{t}b)\|_{p}^p\lesssim 1$. 

Set $u(t,x)=D_{t}b(x)$ be the solution of $\divv A \nabla$ for $(t,x) \in \R^{1+n}$ away from the support of $b$ (identifying $\R^n$ with $\{0\}\times \R^n$)   given by the double layer integral  in \eqref{eq:DLint} (which will be shown to converge  under the De Giorgi  assumption on $\divv A \nabla$ and its adjoint). Under these assumptions,  we know that  $u(t,x)=\mD_{t}b(x)$ for $t>0$  (Proposition \ref{prop:rosen}) where $\mD_{t}$ is the abstract double layer operator.  
We claim that  
\begin{equation}
\label{eq:size}
|u(t,x)|\lesssim |(t,x)|^{-n+1-\mu_{0}}, \ |(t,x)|\ge 2. 
\end{equation}
Indeed, using \eqref{eq:DLint},   Proposition 2.1 in \cite{AAAHK} for the solution $(s,y) \mapsto \Gamma^{A^*}( s,y; t,x)$ for 
$L^*$ in $(-2,2)\times \Delta(0,2)$, as $\divv A^*\nabla $ is $t$-independent (this results extends \textit{mutatis mutandi} to systems), and then \eqref{eq:gammagrad}, we have
\begin{align*}
\label{}
   |u(t,x)| &  \le \|A^*\|_{\infty}\|b\|_{2}\left(\int_{\Delta(0,1)} |(\nabla_{s,y}\Gamma^{A^*})( 0,y; t,x)|^2\, dy\right)^{1/2} \\
    &  \lesssim \left( \int_{\Delta({0},1)}\int_{-1}^1 |(\nabla_{s,y}\Gamma^{A^*})( s,y; t,x)|^2\, dsdy\right)^{1/2}
    \\
    &
    \lesssim   |(t,x)|^{-n+1-\mu_{0}}.
\end{align*}
Remark that a similar strategy gives (bad but finite) pointwise bounds for $D_{t}b(x)$ for $(t,x)$ not in the support of $b$, showing that $u$ is well-defined. As we shall see, \eqref{eq:size} is all we need to run the Kenig-Pipher method.  

Fix $k\ge 2$, $x\in C_{k}$. We estimate $ t^{-(1+n)}{}\int_{W(t,x)} |\nabla u|^2$. If $t\ge 2^k$, by  Caccioppoli inequality and \eqref{eq:size}
$$ 
 t^{-(1+n)}{}\int_{W(t,x)} |\nabla_{A} u|^2 \le C t^{-(3+n)}{}\int_{\widetilde W(t,x)} |u|^2 \lesssim 2^{-2k(n+\mu_{0})},
$$ 
with $\widetilde W(t,x)$ a slightly enlarged version of $W(t,x)$.  

It remains to consider the case $t<2^k$.  The  argument of \cite{KP}, Lemma 8.10, p.~494, yields the estimate  for some $C$ depending only on ellipticity and dimension, 
$$ 
\sup_{t<2^k} t^{-(1+n)}{}\int_{W(t,x)} |\nabla_{A} u|^2 \le C \sup_{t\le (1+c_{0})2^k, x\in \widetilde C_{k}} |\pd_{t}u(t,x)|^2 + C M(|\nabla_{x}u(0,\cdot)|^q 1_{\widetilde C_{k}})(x)^{2/q},
 $$  
  for some $q<2$ (coming from usage of Poincar\'e inequalities in $\R^n$), where $M$ is the Hardy-Littlewood maximal operator and $\widetilde C_{k}$ is the union of all $\Delta(x,c_{1}t)$ for $x\in C_{k}$ and $t<2^k$. Thus if $c_{1}$ in the definition of $W(t,x)$ is chosen small to start with, $\widetilde C_{k}$ is an annulus  at distance proportional to $2^k$ from the support of $b$ of the form $c_{3}2^k \le |x|\le c_{4}2^k$. 
  As $A$ is $t$-independent, we have that $\pd_{t}u$ is also a solution. Moser's local estimate \eqref{eq:Mos}, Caccioppoli inequality and \eqref{eq:size} imply that 
   \begin{align*}
\label{}
    \sup_{t\le (1+c_{0})2^k, x\in \widetilde C_{k}} |\pd_{t}u(t,x)|^2 &
    \lesssim  2^{-k(n+1)}\int_{\check C_{k}}\int_{-(1+c_{0})2^{k+1/2}}^{(1+c_{0})2^{k+1/2}} |\pd_{t}u(t,x)|^2\, dtdx
    \\ 
    &\lesssim 2^{-k(n+3)} \int_{\check C_{k}}\int_{-(1+c_{0})2^{k+1}}^{(1+c_{0})2^{k+1}} |u(t,x)|^2\, dtdx 
    \\
    & \lesssim    2^{-2k(n+\mu_{0})},
      \end{align*}
 where  $\widetilde C_{k} \subsetneq \check C_{k} \subsetneq \hat C_{k}$ are annuli of the form $|x|\sim 2^k$. 
  
  For the last term, we have by the Hardy-Littlewood theorem,
 $$\| M(|\nabla_{x}u(0,\cdot)|^q 1_{\widetilde C_{k}})(x)^{2/q}\|_{L^2(C_{k})}^2 \le C \int_{\widetilde C_{k}} |\nabla_{x}u(0,\cdot)|^2\, dx,$$
 and since $u$ is a weak solution of $L$ away from the support of $b$ and $A$ has $t$-independent coefficients, we have by Proposition 2.1 in \cite{AAAHK} and then Caccioppoli inequality
 \begin{align*}
 \int_{\widetilde C_{k}} |\nabla_{x}u(0,x)|^2\, dx &
 \lesssim 2^{-k}\int_{\widetilde C_{k}}\int_{-2^k}^{2^k} |\nabla_{t,x}u(t,x)|^2\, dtdx 
 \\
 &
 \le 2^{-3k}\int_{\check C_{k}}\int_{-2^{k+1}}^{2^{k+1}} |u(t,x)|^2\, dtdx 
 \\
 &
 \lesssim 2^{-k(n+2\mu_{0})},
  \end{align*}
  where $\check C_{k}$ is again a slightly larger version of $\widetilde C_{k}$.
  
  Gathering all the estimates we have obtained that $\|\tN (\nabla_{A} \mD_{t}b)\|_{L^2(C_{k})} \le C2^{-k(n/2 +\mu_{0})}$ as desired. 
  
  Let us present the proof for the single layer.  By \cite{MSV}, Theorem 4.1,  if $p=1$, and \cite{YZ}  if $p \le 1$,  and interpolation,  it is enough to prove the bound for $2$-atoms for $H^p(\R^n;\C^m)$ for $p$ as above.  

Let  $a$ be a $2$-atom for  $H^p(\R^n;\C^m)$, with $a$  supported in a surface ball $\Delta(x_{0},r)$.  
Again, we assume  that $x_{0}=0$ and $r=1$ to simplify the exposition. We let $\Delta_{k}=\Delta(0,2^k)$ and $C_{k}=\Delta_{k+1}\setminus \Delta_{k}$ for $k\in \N$.  We have
\begin{align*}
\|\tN (\nabla_{A} \mS_{t}a)\|_{L^p(\Delta_{2})} &\le |\Delta_{2}|^{1/p-1/2}  \|\tN (\nabla_{A} \mS_{t}a)\|_{L^2(B_{2})}
\\ &\le C \|a\|_{2} |\Delta_{2}|^{1/p-1/2} \le C 4^{n(1/p-1/2)}.
\end{align*}

As above, it is  enough to show   for $u(t,x)=S_{t}a(x)$, which is  is a weak solution of $\divv A \nabla$ for $(t,x) \in \R^{1+n}$ away from the support of $a$,  the estimate
\begin{equation}
\label{eq:size2}
|u(t,x)|\lesssim |(t,x)|^{-n+1-\mu_{0}},  \ |(t,x)|\ge 2.
\end{equation}
Indeed, we know from Proposition \ref{prop:rosen} that $\nabla_{A}\mS_{t}a=\nabla_{A}S_{t}a$ for $t>0$. By the mean value of $a$ we can write 
$$
u(t,x)= \int_{\Delta(0,1)} (\Gamma^A( t,x; 0,y) - \Gamma^A( t,x; 0,0)) a(y)\, dy$$ and conclude using \eqref{eq:gammaholder}.
  \end{proof}

\section{Extrapolation of solvability for regularity and Neumann problems}\label{sec:extra}

We are now ready to attack the extrapolation for solvability by gathering all pieces of information obtained so far. 

Let $\divv A \nabla $ be a uniformly  complex elliptic system with $A(\bx)$ measurable, bounded, $t$-independent on $\R^{1+n}$ with the strict G\aa rding inequality     on $\mH^0$, namely \eqref{eq:accrassumption}.

In addition, assume that $\divv A \nabla $ and its adjoint satisfy the De Giorgi condition. Consider the reflected matrix $A^\sharp$. It is no longer with $t$-independent coefficients. However, it is easy to see that it does satisfy the G\aa rding inequality  \eqref{eq:garding} on $\ree$. We also assume that the second order system with matrix $ A^\sharp $ and its adjoint satisfy the De Giorgi condition.  We call $\mu_{DG}\in (0,1]$ the best exponent that works for all 4 operators. 

With these conditions, all results in prior sections apply. Again, this situation covers   dimension $1+n=2$   or dimensions $1+n\ge 3$ with $A$  close in $L^\infty$ to a real and scalar matrix.

Here is a fact we are going to use. Let $\frac{n}{n+1}<p\le 2$. It is shown in \cite{HMiMo} (Lemma 6.1) that any  weak solution $u$ of $\divv A \nabla u=0$
with $\|\tN(\nabla_{A}u)\|_{p}<\infty$ admits a conormal gradient at the boundary in $H^p(\R^n; (\C^m)^{1+n})$  with $\|\nabla_{A}u|_{t=0}\|_{H^p
} 
\lesssim \|\tN(\nabla_{A}u)\|_{p}$ and  that $\nabla_{A}u(t,\cdot)$ converges in the sense of distributions to $\nabla_{A}u|_{t=0}$ as $t\to 0$.  The implicit constant depends only on the $L^\infty$ bound for $A$ and dimension.   In particular, if $u$ is also an energy solution, then the two notions of conormal gradients at the boundary must coincide from the convergence in the sense of distributions.

As mentioned in the introduction,  we shall restrict our attention to solvability exponents not exceeding 2. See \cite{HKMP2} for the Regularity problem for exponents exceeding 2.    See also the forthcoming \cite{AS}. 
 
\subsection{Regularity problem}

Slightly modifying  the original approach of \cite{KP}, we say that the Regularity problem $ (R^p_{A})$ is solvable if  there exists $C_{p}<\infty$ such that
for any $f\in H^p_{\nabla}(\R^n; (\C^m)^n)\cap \dot H^{-1/2}_{\nabla}(\R^n; (\C^m)^n)$ 
the energy solution $u$ of $\divv A \nabla u=0$ with regularity data $\nabla_{x}u|_{t=0}=f $ satisfies 
$$
\|\tN(\nabla_{A}u)\|_{p}\le C_{p} \|f\|_{H^p_{\nabla}(\R^n; (\C^m)^n)}.
$$
There is a difference between solvability and well-posedness in the class where $\|\tN(\nabla_{A}u)\|_{p}<\infty$ given the data $f$. See the discussion in \cite{HKMP2} about uniqueness. Axelsson \cite{Ax} also showed by an explicit example for a real equation in dimension $1+n=2$ that there might be solutions not in the energy class, even for very smooth data, while the energy solution does not satisfy this bound.  

Solvability implies well-posedness of the following restricted problem: given $f\in H^p_{\nabla}(\R^n; (\C^m)^n)$, there exists a unique solution $u$ of $\divv A \nabla u=0$ with $
\|\tN(\nabla_{A}u)\|_{p}\le C_{p} \|f\|_{H^p_{\nabla}(\R^n; (\C^m)^n)}$,
$\nabla_{x}u|_{t=0}=f$ and such that there exists a sequence of energy solutions $u_{k}$ with 
$\|\tN(\nabla_{A}u- \nabla_{A} u_{k})\|_{p} \to 0$. The constant $C_{p}$ is the one specified by solvability assumption. This follows from density of $H^p_{\nabla}(\R^n; (\C^m)^n)\cap \dot H^{-1/2}_{\Delta}(\R^n; (\C^m)^n)$ in $H^p_{\nabla}(\R^n; (\C^m)^n)$. This fact, which is just a reformulation of the  extension by continuity for linear maps, is left to the reader. 

\begin{thm}\label{thm:extrareg} Assume that $A$ is as specified at the beginning of the section with $\mu_{DG}\in (0,1]$.  Let $p_{DG}= \frac{n}{n+\mu_{DG}}$.  Let $1<r\le 2$. Assume that  the Regularity problem $(R^r_{A})$ is solvable. Then the Regularity problem $(R^p_{A})$ is solvable for $p_{DG}<p<r$.\end{thm}

\begin{cor}\label{cor:postivereg} This theorem applies to the following situations for $A$ in addition to having the assumption at the beginning of the section ($t$-independence, ellipticity and De Giorgi conditions): \begin{enumerate}
  \item  $A$ is  constant  plus $t$-independent $L^\infty$ perturbation
  \item $A$  is hermitian plus $t$-independent $L^\infty$ perturbation
  \item $A$ is block upper-triangular plus $t$-independent $L^\infty$ perturbation
  \item $A$ real (non necessarily symmetric) and scalar plus $t$-independent $L^\infty$ perturbation
\end{enumerate}
\end{cor}

\begin{proof} 
We know from \cite{A} that De Giorgi assumption is stable under $L^\infty$ perturbations. It suffices to show that $(R^r_{A})$ is solvable in the four items for some $1<r\le 2$. 
 From \cite{AAMc}, we also know that $(R^2_{A})$ is stable under $t$-independent $L^\infty$ perturbation of $A$ and is verified for $A$ constant or hermitian (for real symmetric scalar $A$, this was done in \cite{KP}), while \cite{AMcM} proves $(R^2_{A})$ for $A$  block upper-triangular (the block diagonal case is a direct consequence of \cite{AHLMcT}). Hence, the first three items  satisfy $(R^2_{A})$. The fourth item is shown on combining \cite{KKPT, KR} and \cite{B} (who also shows $(R^1_{A})$) if $n+1=2$ and \cite{HKMP2} for $n+1\ge 3$. 
\end{proof}

\begin{rem}
The block upper-triangular case can be slightly relaxed. Instead of the lower coefficient $c$ to be 0, we may only assume $\divv c=0$. See \cite[Remark 6.7]{AMcM}.\end{rem}

\begin{lem}\label{lem:equiv} Let $p_{DG}<p\le 2$. Then $(R^p_{A})$ is solvable if and only if there exists $C_{p}<\infty$ such that for any $u\in \mE$ solution of $\divv A  \nabla u=0$, \begin{equation}
\label{eq:1}
\|\pd_{\nu_{A}}u|_{t=0}\|_{H^p(\R^n; \C^m)}\le C_{p}\|\nabla_{x} u|_{t=0}\|_{H^p_{\nabla}(\R^n; (\C^m)^n)}.
\end{equation}
\end{lem}

\begin{proof} Let $u$ be the energy solution with regularity data $f=\nabla_{x} u|_{t=0}$. Let $\alpha=\pd_{\nu_{A}}u|_{t=0}$.  The solvability of $(R^p_{A})$, together with 
$\|\nabla_{A}u|_{t=0}\|_{H^p(\R^n;(\C^m)^{1+n})} \lesssim \|\tN(\nabla_{A}u)\|_{p}$, implies the desired inequality. Conversely, assuming this estimate for any energy solution, we can use Theorem \ref{thm:main}  to conclude that $
\|\tN(\nabla_{A}u)\|_{p}\lesssim C_{p} \|\nabla_{x} u|_{t=0}\|_{H^p_{\nabla}(\R^n; (\C^m)^n)}$, hence $(R^p_{A})$ is solvable.
\end{proof}

\begin{proof}[Proof of Theorem \ref{thm:extrareg}]  
Let us begin with the case $p_{DG}<p\le 1$. By Lemma \ref{lem:equiv} and Theorem \ref{thm:boundaryregimprovhardy}, it suffices to show that if
  $a$  is a 2-atom for $H^p_{\nabla}(\R^n; (\C^m)^n)$, then we obtain a uniform estimate $\|\alpha\|_{H^p(\R^n;\C^m)}\le C$ where $\alpha$ is the conormal derivative of the energy solution $u$ produced by the Dirichlet datum $b$ with $a=\nabla b$ as in the definition of 2-atoms for $H^p_{\nabla}(\R^n;\C^m)$. By scale invariance of our assumptions, we assume   that  $a$ and $b$ are supported in the surface ball $\Delta(0,1)$.   We shall show that $\alpha$ is a $r$-molecule for $H^p(\R^n;\C^m)$ (see below) with bound independent of $a$.  Hence there is a constant $C$ independent of $a$ such that $\|\alpha\|_{H^p(\R^n;\C^m)}\le C$ as desired.  
 
We now prove the $r$-molecule property for $\alpha$. As in the proof   of Theorem \ref{thm:main2}, 
 set  $\Delta_{k}=\Delta(0,2^k)$ and $C_{k}=\Delta_{k+1}\setminus \Delta_{k}$ for $k\in \N$.
 It suffices to show that $\|\alpha\|_{L^r(\Delta _{2})} \lesssim 1$ and that $\|\alpha\|_{L^r(C_{k})} \lesssim 2^{-k\mu} 2^{-nk/r'}$ for some $0<\mu<\mu_{DG}$ with $p>\frac{n}{n+\mu}$ with $r'$ the conjugate exponent to $r$. Indeed, 
 $2^{-k\mu} 2^{-nk/r'}= 2^{-k\varepsilon} 2^{-nk(1/p-1/r)}$ for $\varepsilon=-\mu+ n(1/p-1)>0$ which is  the right decay for being in the Hardy space $H^p$.  The local estimate $\|\alpha\|_{L^r(\Delta _{2})} \lesssim 1$ follows from the global bound $\|\alpha\|_{r}\lesssim \|a\|_{r} \lesssim \|a\|_{2} \lesssim 1$ (here we use that the support of $a$ is contained in  $\Delta_{0}=\Delta(0,1)$). 
The main task is therefore to obtain the decay on $C_{k}$. Note that $C_{k}$ can be covered by boundedly (in $k$) many surface balls $\Delta$ with radius proportional to $2^k$ and with distance to $ \Delta_{0}$ proportional to $2^k$ and with $4\Delta \cap \Delta_{0}=\emptyset$. Thus it is enough to work on one of those. Let $g\in C_{0}^\infty(\Delta; \C^m)$ with $\|g\|_{r'} \le 1$. It suffices to estimate $\langle \alpha, g\rangle$. Let $w$ be the energy solution of $\divv A^*\nabla w=0$ on $\reu$ with $w|_{t=0}=g$ (Lemma \ref{lemma2}). Using Theorem \ref{thm:boundaryreg}, we deduce from $(R^r_{A})$ that
$$
\|\pd_{\nu_{A^*}}w|_{t=0}\|_{\dot W^{-1,r'}} \le C_{p'}\|\nabla g\|_{\dot W^{-1,r'}} \lesssim \|g\|_{r'}\le 1.$$
We deduce from the  representation of Corollary \ref{cor:BL} that $\sup_{t>0} \|w(t,\cdot)\|_{r'}\lesssim 1$. 
We now invoke \eqref{eq:localboundaryreg}, which tells
$$
|\langle \alpha, g \rangle| \le C  2^{-2k}\left({}\int_{\Omega_{+}} |u|^2\right)^{1/2}
\left({}\int_{\Omega_{+}} |w|^2\right)^{1/2},
$$
with $\Omega_{+}$  contained in a box $[0,c2^k]\times 3\Delta$. H\"older's inequality using $r'\ge 2$ yields
$$
\left({}\int_{\Omega_{+}} |w|^2\right)^{1/2} \le |\Omega_{+}|^{1/2-1/r'}  \left(\int_{0}^{c2^k} \|w(t,\cdot)\|^{r'}\right)^{1/r'} \lesssim 2^{k/r'} 2^{(1+n)k(1/2-1/r')}.
$$
Now for $u$ we use the decay estimate from Lemma \ref{lemma3} together with the observation that $ \|b\|_{\dot H^{1/2}}\lesssim (\|b\|_{2}\|\nabla b\|_{2})^{1/2}\lesssim 1$, to get $|u|\le 2^{-k(n-1+\mu)}$ on $\Omega_{+}$.  Working out the powers of $2^k$ we obtain the desired bound for $|\langle \alpha, g \rangle|$.

We now continue with the case $1<p<r$. Another way to reformulate Lemma \ref{lem:equiv}
is to say that the  Dirichlet to Neumann operator satisfies $ \|\Gamma_{DN}f\|_{H^r(\R^n; \C^m)}\le C_{r}\|f\|_{H^r_{\nabla}(\R^n; (\C^m)^n)}$ for all $f\in H^r_{\nabla}(\R^n; (\C^m)^n) \cap \dot H^{-1/2}_{\nabla}(\R^n; (\C^m)^n)$. Let $T_{r}$ be the  continuous extension from $H^r_{\nabla}(\R^n; (\C^m)^n)$ into $H^r(\R^n; \C^n)$.  We just showed a uniform estimate for $ \|\Gamma_{DN}a\|_{H^1(\R^n; \C^m)}$ when $a$ is a 2-atom for $H^1_{\nabla}(\R^n;\C^m)$, which are elements in $H^r_{\nabla}(\R^n; (\C^m)^n) \cap \dot H^{-1/2}_{\nabla}(\R^n; (\C^m)^n)$. Hence $T_{r}a=\Gamma_{DN}a$. We can apply  the same interpolation procedure as for (4) of Proposition \ref{lem:hpgrad}. Hence,  for $1<p<r$, we obtain
$T_{r}$ bounded from  $H^p_{\nabla}(\R^n; (\C^m)^n)$ into $H^p(\R^n; \C^m)$. In particular, we obtain  $ \|\Gamma_{DN}f\|_{H^p(\R^n; \C^m)}\le C_{p}\|f\|_{H^p_{\nabla}(\R^n; (\C^m)^n)}$ for all $f\in \mD_{\nabla}(\R^n;(\C^m)^n)$, that is 
$\|\pd_{\nu_{A}}u|_{t=0}\|_{H^p(\R^n; \C^m)}\le C_{p}\|\nabla_{x} u|_{t=0}\|_{H^p_{\nabla}(\R^n; (\C^m)^n)}$ for all energy solutions with smooth Dirichlet data. We conclude using (1) in Theorem
 \ref{thm:boundaryregimprov} to waive the restriction on the data and then Lemma \ref{lem:equiv} again.
 \end{proof}
  
  \begin{rem} When $p<1$, the solvability information is used to obtain the decay for $\alpha$ but in a dual way, not on the solution $u$ attached to $\alpha$. Note that this argument has the flavor of many of the different steps for Theorem 5.2 of  \cite{KP}. But the order in which they are invoked is completely different trying to use a priori estimates as much as possible. In particular,  we avoid the localization technique there and the recourse to solvability of dual Dirichlet problem \textit{per se}. We only use available \textit{a priori} estimates. This last point will be important later. 
\end{rem}

\subsection{Neumann problem}

Slightly modifying  the original approach of \cite{KP}, we say that the Neumann problem $(N^p_{A})$ is solvable if  there exists $C_{p}<\infty$ such that
for any $g\in H^p(\R^n; \C^m)\cap \dot H^{-1/2}(\R^n; \C^m)$ 
the (modulo constants) energy solution $u$ of $\divv A \nabla u=0$ with conormal derivative $\pd_{\nu_{A}}u|_{t=0}=h$ satisfies
$$
\|\tN(\nabla_{A}u)\|_{p}\le C_{p} \|h\|_{H^p(\R^n; \C^m)}.
$$
Here too, there is a difference between solvability and well-posedness in the class where $\|\tN(\nabla_{A}u)\|_{p}<\infty$ given the data $f$.    An explicit example for a real equation in dimension $1+n=2$ in \cite{Ax} shows that there might be solutions not in the energy class, even for very smooth data, while the energy solution does not satisfy this bound.  Note that in 2 dimensions, Neumann and Regularity problems are the same up to taking conjugates. 

Solvability implies  well-posedness of the following restricted  problem: given $h\in H^p(\R^n; \C^m)$, there exists a unique solution $u$ of $\divv A \nabla u=0$ with $
\|\tN(\nabla_{A}u)\|_{p}\le C_{p} \|h\|_{H^p_{\nabla}(\R^n; (\C^m)^n)}$,
$\pd_{\nu_{A}}u|_{t=0}=h$ and such that there exists a sequence of energy solutions $u_{k}$ with 
$\|\tN(\nabla_{A}u- \nabla_{A} u_{k})\|_{p} \to 0$. The constant $C_{p}$ is the one specified by solvability assumption. This follows from density of $H^p(\R^n; \C^m)\cap \dot H^{-1/2}(\R^n; \C^m)$ in $H^p(\R^n; \C^n)$. This fact is left to the reader. 

\begin{thm}\label{thm:extraneu} Assume that $A$ is as specified at the beginning of the section with $\mu_{DG}\in (0,1]$.  Let $p_{DG}= \frac{n}{n+\mu_{DG}}$.  Let $1<r\le 2$. Assume that  the Neumann problem $(N^r_{A})$ is solvable. Then the Neumann problem $(N^p_{A})$ is solvable for $p_{DG}<p<r$.\end{thm}

\begin{cor}\label{cor:postiveneu} This theorem applies to the following situations for $A$ in addition to having the assumption at the beginning of the section ($t$-independence, ellipticity and De Giorgi conditions): \begin{enumerate}
  \item  $A$ is  constant  plus $t$-independent $L^\infty$ perturbation
  \item $A$  is hermitian plus $t$-independent $L^\infty$ perturbation
  \item $A$ is block lower-triangular plus $t$-independent $L^\infty$ perturbation
  \item $A$ real (non necessarily symmetric) and scalar if $1+n=2$ plus $t$-independent $L^\infty$ perturbation
\end{enumerate}
\end{cor}

\begin{proof} It suffices to show that $(N^r_{A})$ is solvable in the four items for some $1<r\le 2$.
We know from\cite{A} that De Giorgi assumption is stable under $L^\infty$ perturbations.   From \cite{AAMc}, we also know that $(N^2_{A})$ is stable under $t$-independent $L^\infty$ perturbation of $A$ and is verified for $A$ constant or hermitian (for real symmetric scalar $A$, this was done in KP), while \cite{AMcM} proves $(N^2_{A})$ for $A$  block upper-triangular (the block diagonal case is a direct consequence of \cite{AHLMcT}). Hence, the first three items  satisfy $(N^2_{A})$. The fourth item is shown on combining \cite{KR} and \cite{B} (who also shows $(N^1_{A})$) as $n+1=2$. 
\end{proof}

We note that solvability of the Neumann problems for real (non-symmetric) equations is still open in dimensions $1+n\ge 3$.

\begin{lem} Let $p_{DG}<p\le 2$. Then $(N^p_{A})$ is solvable if and only if there exists $C_{p}<\infty$ such that for any $u\in \mE$ solution of $\divv A  \nabla u=0$, $\|\nabla_{x}u|_{t=0}\|_{H^p(\R^n; (\C^m)^n)}\le C_{p}\|\pd_{\nu_{A}}  u|_{t=0}\|_{H^p_{\nabla}(\R^n; \C^m )}$.
\end{lem}

Same proof as Lemma \ref{lem:equiv}.

\begin{proof}
[Proof of Theorem \ref{thm:extraneu}] The case $1<p<r$ is exactly as in the proof of Theorem \ref{thm:extrareg} once we have done the case $p_{DG}<p\le 1$.

Let us assume  $p_{DG}<p\le 1$. By Lemma \ref{lem:equiv} and Theorem \ref{thm:boundaryregimprovhardy}, it suffices to show if 
  $a$  is a 2-atom for $H^p(\R^n; \C^n)$, then we obtain a uniform estimate $\|f\|_{H^p_{\nabla}(\R^n;(\C^m)^n)}\le C$ where $f$ is the tangential derivative of any energy solution $u$ produced by the Neumann datum $a$. By scale invariance of our assumptions, we assume   that $a$ is supported in the surface ball $\Delta(0,1)$.   We shall show that $f$ is a $r$-molecule for $H^p(\R^n;(\C^m)^n)$  with bound independent of $a$.  Hence there is a constant $C$ independent of $a$ such that $\|f\|_{H^p(\R^n;(\C^m)^n}\le C$ as desired. Since $f$ is of a gradient form, it automatically fulfills  the $H^p_{\nabla}(\R^n;(\C^m)^n)$ estimate. 
 
We now prove the $r$-molecule property for $f$. As in the proof of  of Theorem \ref{thm:main2}, 
 set  $\Delta_{k}=\Delta(0,2^k)$ and $C_{k}=\Delta_{k+1}\setminus \Delta_{k}$ for $k\in \N$.
 As before, it  suffices to show that $\|f\|_{L^r(\Delta _{2})} \lesssim 1$ and that $\|f\|_{L^r(C_{k})} \lesssim 2^{-k\mu} 2^{-nk/r'}$ for some $0<\mu<\mu_{DG}$ with $p>\frac{n}{n+\mu}$ with $r'$ the conjugate exponent to $r$.   The local estimate $\|f\|_{L^r(\Delta _{2})} \lesssim 1$ follows from the global bound $\|f\|_{r}\lesssim \|a\|_{r} \lesssim \|a\|_{2} \lesssim 1$ (here we use that the support of $a$ is contained in  $\Delta_{0}=\Delta(0,1)$). 
The main task is therefore to obtain the decay on $C_{k}$. Again it is enough to work on a surface ball $\Delta$ with radius proportional to $2^k$ and with distance to $ \Delta_{0}$ proportional to $2^k$ and with $4\Delta \cap \Delta_{0}=\emptyset$.  Let $g\in C^\infty_{0}(\Delta; (\C^m)^n)$ with $\|g\|_{r'} \le 1$. It suffices to estimate $\langle f, g\rangle$.  Write $f=\nabla_{x} u|_{t=0}$ with $u|_{t=0}$ being   the still unspecified Dirichlet data   as we have not yet chosen the constant of integration.   Let $w$ be one of the energy solutions of $\divv A^*\nabla w=0$ on $\reu$ with $\pd_{\nu_{A^*}}w|_{t=0}=-\divv g$ (Lemma \ref{lemmaN}). We have, therefore, $ \langle f, g\rangle= \langle u|_{t=0}, \pd_{\nu_{A^*}}w|_{t=0}\rangle  $. Using Theorem \ref{thm:boundaryneu}, we deduce from $(N^r_{A})$ that
$$
\|\nabla_{x}w|_{t=0}\|_{\dot W^{-1,r'}} \le C_{r'}\|-\divv g\|_{\dot W^{-1,r'}} \sim \|g\|_{r'}\le 1.$$
Using Lemma \ref{lemma:inf}, we now choose $w$ so that $w|_{t=0}\in L^{r'}$. 
We can deduce from the boundary layer representation of Corollary \ref{cor:BL} that $\sup_{t>0} \|w(t,\cdot)\|_{r'}\lesssim 1$.  Next, 
 invoke \eqref{eq:localboundaryneu}, which tells
$$
|\langle u|_{t=0}, \pd_{\nu_{A^*}}w|_{t=0} \rangle| \le C  2^{-2k}\left({}\int_{\Omega_{+}} |u|^2\right)^{1/2}
\left({}\int_{\Omega_{+}} |w|^2\right)^{1/2},
$$
with $\Omega_{+}$  contained in a box $[0,c2^k]\times 3\Delta$.  H\"older's inequality using $r'\ge 2$ yields
$$
\left({}\int_{\Omega_{+}} |w|^2\right)^{1/2} \le |\Omega_{+}|^{1/2-1/r'}  \left(\int_{0}^{c2^k} \|w(t,\cdot)\|^{r'}\right)^{1/r'} \lesssim 2^{k/r'} 2^{(1+n)k(1/2-1/r')}.
$$
Remark that we have not yet specified the constant of integration for $u$.  We choose it now so as to use the decay estimate from Lemma \ref{lem:sizeNeumann}  to get $|u|\le 2^{-k(n-1+\mu)}$ on $\Omega_{+}$ since $\|a\|_{1}\lesssim 1$.  Working out the powers of $2^k$ we obtain the desired bound for $|\langle f, g \rangle|$.
\end{proof}

\begin{rem}
In the block lower-triangular case, when the upper coefficient  $b$ to be 0, or equivalently that the conormal vector field is proportional to the transversal vector field, one can obtain an $L^2$-molecular decay for the tangential gradient by a direct integration by parts which does not use at all the initial $L^2$ solvability information. This one is only used for the local estimate. This means that the difficulty in the study of Neumann problems lies in the upper coefficient of $A$. It remains to understand its exact role when it is not 0 or when $A$ is not hermitian in dimensions $1+n \ge 3$. \end{rem}

\section{Extrapolation of solvability for Dirichlet problems and other Neumann problems}\label{sec:extraDir}

We gather in this section the needed results to prove extrapolation of Dirichlet problems and of a new type of problems, namely Neumann problems with data in  negative Sobolev spaces. 

It is convenient to introduce the following correspondences of spaces to be read line by line. 

\begin{table}[htdp]
\begin{center}

\begin{tabular}{|c|c|c|c|c|c|}
\hline
exponents & $Y$ &$ Y^{-1}$ & $\mT$ & $  X^{1}$ & $X$ \\
\hline
$1<p,q<\infty$, $q=p'$ & $L^q$ &  $\dot W^{-1, q}$ & $T^q_{2}$ &$ \dot W^{1,p} $& $L^p$ \\
\hline
$\alpha=0$, $p=1$ &  BMO & $\dot {\rm BMO}^{-1}$ & $T^\infty_{2}$ &  $ \dot H^{1,1} $ & $H^1$ \\
\hline
$0<\alpha=n( \frac{1}{p}-1) <1$ & $\dot \Lambda^{\alpha}$ & $\dot \Lambda^{\alpha-1}$ & $T^{\infty}_{{2,\alpha}}$ &$ \dot H^{1,p} $ & $H^p$ \\
\hline
\end{tabular}
%\caption{default}
\end{center}

\label{default}
\end{table}%

Here $Y$, $Y^{-1}$ are the dual spaces of $X$, $ X^{1}$ respectively. They are spaces on the boundary. Next, $\mT$ are tent spaces on $\reu$. For $1<q\le \infty$, $T^{q}_{2}$ is the  tent space   of \cite{CMS}. For $q=\infty$, this is defined via Carleson measures: 
$$
\iint_{(0,r)\times \Delta}   |f(t,x)|^2\, \frac{dtdx}{t} \le \|f\|_{T^\infty_{2}}^2 |\Delta|.
$$
For $0<\alpha<1$, 
$$
\iint_{(0,r)\times \Delta}   |f(t,x)|^2\, \frac{dtdx}{t} \le \|f\|_{T^\infty_{2,\alpha}}^2 |\Delta|^{1+\frac{2\alpha}{n}}.
$$
Here $\Delta$ are balls  in $\R^n$ and $r$ is the radius of $\Delta$.  

We next turn to equivalence of boundary norms with  interior estimates of tent space nature.

\begin{thm}\label{thm:intDir}
Let $\divv A \nabla $ be a uniformly  elliptic system with $A(x)$ measurable, bounded, $t$-independent, complex coefficients on $\R^{1+n}$ with the strict G\aa rding inequality     on $\mH^0$, namely \eqref{eq:accrassumption}.
Assume that $\divv A \nabla $ and $\divv A^* \nabla $ satisfy the De Giorgi condition and call $0<\mu_{DG}$ the exponent that works for both. Then   for   all  spaces in the table with $2\le q$ and $\alpha<\mu_{DG}$ and for  any weak solution of $Lu=0$ on the upper half-space $\R^{1+n}_{+}, 1+n\ge 2$, in any of the classes $\mE_{s}$, $-1\le s\le 0$, we have
\begin{equation}\label{eq:mainDir}
\| t\nabla u\|_{\mT} \approx \|\pd_{\nu_A}u|_{t=0}\|_{Y^{-1}}+\|\nabla_{x}u|_{t=0}\|_{Y^{-1}}. 
\end{equation}
  \end{thm}
  
  Again, we do not consider the case $2-\varepsilon<q<2$ which can be handled 
  without the De Giorgi condition (See the forthcoming \cite{AS}). 
  
  \begin{proof}
The inequality $\lesssim$ follows from the generalized boundary layer representation of Corollary \ref{lem:GBLP} together with the estimates proved in \cite{HMaMo} for the single and double layer potentials (again in the case of equations and $1+n\ge 3$ but  with immediate extension to our situation). The converse inequality is a result from  \cite{AS}, where the other direction is proved as well in this generality. 
\end{proof}
  
  \begin{rem}
Remark that in the case  $2<q<\infty$,  the inequality $\gtrsim$ is akin to the inequality (3.9) in \cite{HKMP2}. It is more precise though as it does not contain any non-tangential maximal function. In fact, \cite{AS} will show under the above assumptions that  the non-tangential maximal function of $u$ is controlled in $L^q$ by the  $T^q_{2}$ norm of $t\nabla u$. This was proved for real equations in \cite{HKMP1}. 
\end{rem} 

\subsection{The Dirichlet problem}

Let $Y=Y(\R^n;\C^m)$ be one of the spaces from the above table. We say that the Dirichlet problem $ (D^Y_{A})$ is solvable if  there exists $C_{Y}<\infty$ such that
for any $f\in Y\cap \dot H^{1/2}(\R^n; \C^m)$ 
the energy solution $u$ of $\divv A \nabla u=0$ with Dirichlet data $u|_{t=0}=f $ satisfies 
$$
\| t\nabla u\|_{\mT}\le C_{Y} \|f\|_{Y}.
$$
We remark that we formulate here the Dirichlet problem uniquely in term of the tent space estimate. From the remark above, the non tangential maximal estimate comes as a bonus. 

\begin{cor}\label{cor:Dir} For the spaces $Y$ considered  in Theorem \ref{thm:intDir},  we have that  $ (D^Y_{A})$ is solvable if  and only if  there exists $C_{Y^{-1}}<\infty$ such that
for any $u\in \mE$ solution of $\divv A  \nabla u=0$, $\|\pd_{\nu_{A}}u|_{t=0}\|_{Y^{-1}} \le C_{Y^{-1}}\|\nabla_{x} u|_{t=0}\|_{Y^{-1}}$. 
\end{cor}

The proof is a direct consequence of  Theorem \ref{thm:intDir}. As before, the solvability is  reduced to a boundary estimate. We obtain a refined version of the main result of \cite{HKMP2},
{which treats only the cases $X=L^p, Y=L^{p'}, 1<p<2+\varepsilon$, but not the endpoint spaces}. 

\begin{cor}\label{cor:dualDirReg} Consider the spaces $Y$ of  Theorem \ref{thm:intDir} and their preduals $X$. If  $ (R^X_{A})$ is solvable then $(D^Y_{A^*})$ is solvable. The converse holds when $X=L^p$ and $Y=L^q$, $q=p'$ and when $X=H^1$ and $Y=BMO$. \end{cor}

\begin{proof}  The equivalence in the range $X=L^p$, $Y=L^q$ follows from Corollary \ref{cor:Dir},  
Lemma \ref{lem:equiv} and Theorem  \ref{thm:boundaryreg}. The implication in the other cases and the equivalence when $p=1$ follow from  Theorem  \ref{thm:boundaryreg}, Theorem  \ref{thm:boundaryreg2} and Corollary \ref{cor:Dir}. 
\end{proof}

This applies when the conditions of Corollary \ref{cor:postivereg} are satisfied for $A^*$. Details are left to the reader.

Remark that the back and forth proof allows to replace $Y=BMO$ by $Y=VMO$ in the statement.
This fact that the Dirichlet problems for VMO or BMO data are equivalent was also observed in \cite{DKP} for real equations.

To finish we can state the extrapolation result for the Dirichlet problem.  This is only here that we use more assumptions on $A$. 

\begin{thm} Consider an elliptic system with all the assumptions at the beginning of Section \ref{sec:extra}. Let $2\le q<\infty$ and assume $ (D^Y_{A})$ is solvable for $Y=L^q(\R^n;\C^m)$. Then
$ (D^Y_{A})$ is solvable for $Y=L^p(\R^n;\C^m)$, $q<p<\infty$, BMO$(\R^n;\C^m)$,  and $\dot \Lambda^{\alpha}(\R^n;\C^m)$ with $0<\alpha <\mu_{DG}$. 

\end{thm}

\begin{proof} It suffices to combine Corollary \ref{cor:dualDirReg} with Theorem \ref{thm:extrareg}. 
\end{proof}

\subsection{The Neumann problem in negative Sobolev spaces}

Let $Y=Y(\R^n;\C^m)$ be one of the spaces from the above table. We say that the Neumann problem $ (N^{Y^{-1}}_{A})$ is solvable if  there exists $C_{Y^{-1}}<\infty$ such that
for any $f\in Y^{-1}\cap \dot H^{-1/2}(\R^n; \C^m)$ 
the energy solution $u$ of $\divv A \nabla u=0$ with Neumann data $\pd_{\nu_{A}}u|_{t=0}=f $ satisfies 
$$
\| t\nabla u\|_{\mT}\le C_{Y^{-1}} \|f\|_{Y^{-1}}.
$$

\begin{cor}\label{cor:Neu} For the spaces $Y$ considered  in Theorem \ref{thm:intDir},  we have that  $ (N^{Y^{-1}}_{A})$ is solvable if  and only if  there exists $C_{Y^{-1}}<\infty$ such that
for any $u\in \mE$ solution of $\divv A  \nabla u=0$, $\|\nabla_{x} u|_{t=0}\|_{Y^{-1}}   \le C_{Y^{-1}} \|\pd_{\nu_{A}}u|_{t=0}\|_{Y^{-1}}$. 
\end{cor}

The proof is a direct consequence of  Theorem \ref{thm:intDir}. As before, the solvability is  reduced to a boundary estimate.  

\begin{cor}\label{cor:dualNeuNeu} Consider the spaces $Y$  of Theorem \ref{thm:intDir} and their preduals $X$. If  $ (N^X_{A})$ is solvable then $(N^{Y^{-1}}_{A^*})$ is solvable. The converse holds when $X=L^p$ and $Y^{-1}=\dot W^{-1,q}$, $q=p'$ and when $X=H^1$ and $Y^{-1}=BMO^{-1}$. \end{cor}

\begin{proof}  The equivalence in the range $X=L^p$, $Y=L^q$ follows from Corollary \ref{cor:Dir},  
Lemma \ref{lem:equiv} and Theorem  \ref{thm:boundaryneu}. The implication in the other cases and the equivalence when $p=1$ follow from  Theorem  \ref{thm:boundaryneu}, Theorem  \ref{thm:boundaryneu2} and Corollary \ref{cor:Dir}. 
\end{proof}

This applies when the conditions of Corollary \ref{cor:postiveneu} are satisfied for $A^*$. Details are left to the reader. 

Remark that the back and forth proof allows to replace $Y=BMO$ by $Y=VMO$ in the statement.

To finish we can state the extrapolation result for the Neumann problem in negative Sobolev spaces.  This is only here that we use more assumptions on $A$. 

\begin{thm} Consider an elliptic system with all the assumptions at the beginning of Section \ref{sec:extra}. Let $2\le q<\infty$ and assume $ (N^{Y^{-1}}_{A})$ is solvable for $Y^{-1}=\dot W^{-1,q}(\R^n;\C^m)$. Then
$ (N^{Y^{-1}}_{A})$ is solvable for $Y^{-1}=\dot W^{-1,p}(\R^n;\C^m)$, $q<p<\infty$, BMO$^{-1}(\R^n;\C^m)$,  and $\dot \Lambda^{\alpha-1}(\R^n;\C^m)$ with $0<\alpha <\mu_{DG}$. 

\end{thm}

\begin{proof} It suffices to combine Corollary \ref{cor:dualNeuNeu} with Theorem \ref{thm:extraneu}. 
\end{proof}

\bibliographystyle{acm}
%GATHER{AKMcDirac.bib}  % makes sure WinEdt finds citations...
%\bibliography{tdependent}

\begin{thebibliography}{}

\end{thebibliography}


\begin{thebibliography}{10}

\def\AIF{Ann. Inst. Fourier \ }
\def\RMI{Rev. Math. Iberoamericana\ }
\def\SM{Studia Math.\ }
\def\TAMS{Trans. Amer. Math. Soc.\ }
\def\JFA{J. Funct. Anal.\ }
\def\CRAS{C. R. Acad. Sci. Paris\ }
\def\MRL{Math. Research Letters\ }
\def\CM{Colloquium Math.\ }
\def\CPDE{Comm. in Partial and Diff. Eq.\ }

\bibitem[AAAHK]{AAAHK} 
{\sc Alfonseca M.,  Auscher P., Axelsson A., Hofmann S., and Kim, S.}
\newblock { Analyticity of layer potentials and $L^{2}$ Solvability of boundary value problems for divergence form elliptic
equations with complex $L^{\infty}$ coefficients}, 
\newblock Adv.  Maths. 226 (2011) 4533--4606.


\bibitem[A]{A}
{\sc Auscher, P.}
\newblock Regularity theorems and heat kernel for elliptic operators.
\newblock {\em J. London Math. Soc. (2) 54}, 2 (1996), 284--296.


\bibitem[AA]{AA}
{\sc Auscher, P., and Axelsson, A.}
\newblock Weighted maximal regularity estimates and solvability of non-smooth
  elliptic systems {I}.
\newblock Invent. Math. (2011) 184: 47--115.



%\bibitem[AAH]{AAH}
%{\sc Auscher, P., Axelsson, A., and Hofmann, S.}
%\newblock Functional calculus of {D}irac operators and complex perturbations of
%  {N}eumann and {D}irichlet problems.
%\newblock {\em J. Funct. Anal. 255}, 2 (2008), 374--448.

%\bibitem{elAAM}
%{\sc Auscher, P., Axelsson, A., and McIntosh, A.}
%\newblock On a quadratic estimate related to the {K}ato conjecture and boundary
%  value problems.
%\newblock {\em Contemp. Math. 205\/} (2010), 105--129.

\bibitem[AAMc]{AAMc}
{\sc Auscher, P., Axelsson, A., and McIntosh, A.}
\newblock Solvability of elliptic systems with square integrable boundary data.
\newblock {\em Ark. Mat. 48\/} (2010), 253--287.

\bibitem[AHLMcT]{AHLMcT}
{\sc Auscher, P., Hofmann, S., Lacey, M., McIntosh, A., and Tchamitchian, P.}
\newblock The solution of the {K}ato square root problem for second order
  elliptic operators on {${\mathbf R}^n$}.
\newblock {\em Ann. of Math. (2) 156}, 2 (2002), 633--654.



\bibitem[AMcM]{AMcM} {\sc Auscher,  P., M$^{\rm c}$Intosh A. , and Mourgoglou M.}
\ {On $L^2$ solvability of BVPs for elliptic systems}, 
\newblock J. Fourier Anal. Appl. 2013, Volume 19, Issue 3, pp 478-494.

\bibitem[AMcR]{AMcR} 
{\sc Auscher, P., McIntosh, A., and Russ, E.}
\newblock Hardy spaces of differential forms on Riemannian manifolds. 
\newblock {\em J. Geom. Anal.} {\bf 18} 1, (2008), 192-248.

\bibitem[AMcT] {AMcT}
{\sc Auscher, P., McIntosh, A., and Tchamitchian, P.}
\newblock Heat kernels of complex elliptic operators and
applications,  
\newblock {\em \JFA} {\bf 152} (1998), 22--73

\bibitem[ART]{ART} {\sc Auscher, P., Russ, E., and Tchamitchian, P.} 
\newblock Hardy-Sobolev spaces on strongly Lipschitz
domains of $\R^n$, 
\newblock {\em J. Funct. Anal.} 218 (2005), 54--109.
%\bibitem{AMcN}
%{\sc Auscher, P., McIntosh, A., and Nahmod, A.}
%\newblock The square root problem of {K}ato in one dimension, and first order
%  elliptic systems.
%\newblock {\em Indiana Univ. Math. J. 46}, 3 (1997), 659--695.

%\bibitem{AQ}
%{\sc Auscher, P., and Qafsaoui, M.}
%\newblock Equivalence between regularity theorems and heat kernel estimates for
%  higher order elliptic operators and systems under divergence form.
%\newblock {\em J. Funct. Anal.}, 177 (2000), 310--364.

\bibitem[AR]{AR}
{\sc Auscher, P., and Ros\'en, A.}
\newblock Weighted maximal regularity estimates and solvability of non-smooth
  elliptic systems {II}.
\newblock  {\em Analysis and PDE}, Vol. 5 (2012), No. 5, 983--1061.

\bibitem[AS]{AS}
{\sc Auscher, P., and Stahlhut, S.}
\newblock {\textit{A priori} estimates for boundary value elliptic problems}
\newblock In preparation


\bibitem[Ax]{Ax}
{\sc Axelsson, A.}
\newblock Non unique solutions to boundary value problems for non symmetric
  divergence form equations.
\newblock {\em Trans. Amer. Math. Soc. 362\/} (2010), 661--672.




\bibitem[AKMc]{AKMc}
{\sc Axelsson, A., Keith, S., and McIntosh, A.}
\newblock Quadratic estimates and functional calculi of perturbed {D}irac
  operators.
\newblock {\em Invent. Math. 163}, 3 (2006), 455--497.


\bibitem[BB]{BB} 
{\sc Badr, N.,  and Bernicot, F.} 
\newblock Abstract Hardy-Sobolev spaces and Interpolation, 
\newblock  {\em J. Funct. Anal. 259} (2010), no. 5, 1169--1208.

\bibitem[BG]{BG} {\sc Badr, N.,  and  Dafni, G.} 
\newblock An atomic decomposition of the Hajlasz Sobolev space $M^{1}_{1}$ on manifolds,
\newblock  {\em J. Funct. Anal. 259} (2010), no. 6, 1380--1420.


\bibitem[B]{B} {\sc Barton, A.}
\newblock  {Elliptic partial differential equations with complex coefficients}, 
\newblock {\em Mem. Amer. Math. Soc.}, posted on October 24, 2012, PII S 0065-9266(2012)00677-0 (to appear in print).

\bibitem[Be]{Be} 
{\sc Bernal, A.} 
\newblock Some results on complex interpolation of $T^p_{q}$
 spaces, Interpolation spaces and related topics,
 \newblock (Ramat-Gan), Israel Mathematical Conference Proceedings, vol. 5, 1992, pp. 1--10.
 
 \bibitem[Br]{Br} 
 {\sc Brown, R.}
 \newblock The Neumann problem on Lipschitz domains in Hardy spaces of order less than one. \newblock {\em Pacific J. Math. 171} (1995), no. 2, 389?407.

%\bibitem[CaFK]{CaFK}
%{\sc Caffarelli, L., Fabes, E., and Kenig, C.~E.}
%\newblock Completely singular elliptic-harmonic measures.
%\newblock {\em Indiana Univ. Math. J. 30}, 6 (1981), 917--924.
%
%
%\bibitem[CFMS]{CFMS} {\sc L. Caffarelli, E. Fabes, Mortola and  S. Salsa,}
%\newblock  Boundary behavior of nonnegative solutions of elliptic operators in divergence form,
%\newblock  {\em Indiana Univ. Math. J.}  {\bf 30}  (1981), no. 4, 621--640.
 
 \bibitem[CDoK]{CDoK}
 {\sc  Cho, S.,  Dong, H., and Kim, S.}
 \newblock Global Estimates for Green's Matrix of Second Order
Parabolic Systems with Application to Elliptic Systems
in Two Dimensional Domains.
\newblock {\em Pot. Anal. 36} (2012), 339--372.
%\bibitem{C}
%{\sc Campanato, S.}
%\newblock {\em Sistemi ellittici in forma divergenza. {R}egolarit\`a
%  all'interno.}
%\newblock Quaderni. Scuola Normale Superiore Pisa, Pisa, 1980.

%\bibitem{Coif}
%{\sc Coifman, R.~R.}
%\newblock A real variable characterization of {$H^p$}.
%\newblock {\em Studia Math. 51\/} (1974), 269--274.
%
%\bibitem{CF}
%{\sc Coifman, R.~R., and Fefferman, C.}
%\newblock Weighted norm inequalities for maximal functions and singular
%  integrals.
%\newblock {\em Studia Math. 51\/} (1974), 241--250.

\bibitem[CMcM]{CMcM}
{\sc Coifman, R., McIntosh, A., and Meyer, Y.}
 \newblock {L'int\'egrale de Cauchy d\'efinit un op\'erateur born\'e sur
  $L^2$ pour les courbes lipschitziennes}.
\newblock {\em Annals of Math. 116} (1982), 361--387.


\bibitem[CMS]{CMS}
{\sc Coifman, R.~R., Meyer, Y., and Stein, E.~M.}
\newblock Some new function spaces and their applications to harmonic analysis.
\newblock {\em J. Funct. Anal. 62}, 2 (1985), 304--335.

%\bibitem[CMS]{CMS} R.R.\,Coifman, Y.\,Meyer, E.M.\,Stein,  {\it Some
%new function spaces and their applications to harmonic analysis,}
% J. Funct. Anal., 62 (1985), no. 2, 304--335.

\bibitem[CV]{CV} {\sc Cohn, W. S., Verbitsky, I. E.,} 
\newblock { Factorization of tent spaces and Hankel operators}, 
\newblock {\em J. Funct. Anal. 175} (2000), no. 2, 308--329. 


\bibitem[Da]{Da}
{\sc Dahlberg, B.}
\newblock On the absolute continuity of elliptic measures.
\newblock {\em Amer. J. Math. 108}, 5 (1986), 1119--1138.
%
%\bibitem{DJK}
%{\sc Dahlberg, B., Jerison, D., and Kenig, C.}
%\newblock Area integral estimates for elliptic differential operators with
%  nonsmooth coefficients.
%\newblock {\em Ark. Mat. 22}, 1 (1984), 97--108.
%



  
 \bibitem[DaK]{DaK}  {\sc Dahlberg, B., and  Kenig, C.} 
 Hardy spaces and the Neumann problem in $L^p$ for
Laplace's equation in Lipschitz domains. Ann. Math. 125, 437--465 (1987).

\bibitem[DeG]{DeG}
{\sc De~Giorgi, E.}
\newblock Sulla differenziabilit\`a e l'analiticit\`a delle estremali degli
  integrali multipli regolari.
\newblock {\em Mem. Accad. Sci. Torino. Cl. Sci. Fis. Mat. Nat. 3}, 3 (1957),
  25--43.

\bibitem[Di]{Di} 
{\sc Dindos, M.}
\newblock Hardy spaces and potential theory on $C^1$ domains in Riemannian manifolds. 
\newblock {Mem. Amer. Math. Soc. 191} (2008), no. 894, vi+78 pp.

\bibitem[DKP]{DKP} {\sc Dindos, M.,  Kenig, C., and Pipher, J.} 
\newblock BMO solvability and the $A_{\infty}$ condition for elliptic operators,
\newblock {\em Journal of Geometric Analysis 21 no.1} (2011), 78--95.
  
  \bibitem[DK]{DK} {\sc Dindos, M., and Kirsch,  J.}
  \newblock{The regularity problem for elliptic operators with boundary data in Hardy-Sobolev space $HS^1$}, 
  \newblock {\em Math. Res. Lett. 19} (2012), no. 3, 699--717. 
  
\bibitem[DoK]{DoK} {\sc Dong, H., and Kim, S.}
\newblock Green's matrices of second order elliptic systems with measurable coefficients in two dimensional domains
\newblock {\em Trans. AMS. Volume 361, Number 6} (2009), 3303--3323.

%\bibitem[FJK]{FJK}
%{\sc Fabes, E., Jerison, D., and Kenig, C.}
%\newblock Necessary and sufficient conditions for absolute continuity of
%  elliptic-harmonic measure.
%\newblock {\em Ann. of Math. (2) 119}, 1 (1984), 121--141.

%\bibitem{FS}
%{\sc Fefferman, C., and Stein, E.~M.}
%\newblock {$H^p$} spaces of several variables.
%\newblock {\em Acta Math. 129}, 3--4 (1972), 137--193.
%
%\bibitem{Fri}
%{\sc Friedman, A.}
%\newblock {\em Partial differential equations}.
%\newblock R.E. Krieger Publishing Co., Huntington, N.Y., 1983.
%
%\bibitem{Gia2}
%{\sc Giaquinta, M.}
%\newblock {\em Multiple integrals in the calculus of variations and nonlinear
%  elliptic systems}.
%\newblock No.~105 in Annals of Mathematics Studies. Princeton University Press,
%  Princeton, NJ, 1983.
%  
%  \bibitem{GW}
%  {\sc Gr\"uter, Widman}
  
  \bibitem[HKMP1]{HKMP1} {\sc Hofmann, S., Kenig, C., Mayboroda, S.,  and Pipher, J.} {Square function/Non-tangential maximal
estimates and the Dirichlet problem for
non-symmetric elliptic operators,}    arXiv:1202.2405.

\bibitem[HKMP2]{HKMP2} {\sc Hofmann, S., Kenig, C., Mayboroda, S.,  and Pipher, J.} The Regularity problem for second order elliptic operators with complex-valued bounded measurable coefficients,  	arXiv:1301.5209.
  
  \bibitem[HK]{HK} {\sc Hofmann, S., and  Kim, S.}
  \newblock  {The Green function estimates for strongly elliptic systems of second order}, \newblock {\em Manuscripta Math. 124} (2007), no. 2, 139--172.

  
  \bibitem[HMiMo]{HMiMo} {\sc Hofmann, S., Mitrea, M., and Morris, A.}
\newblock { The method of layer potentials in
$L^p$ and endpoint spaces for elliptic operators with $L^\infty$ 
coefficients.}, in preparation.

\bibitem[HMaMo]{HMaMo}  {\sc Hofmann, S., Mayboroda, S., and Mourgoglou, M.}
 \newblock $L^p$ and endpoint solvability results for divergence form elliptic
equations with complex $L^{\infty}$ coefficients, 
in preparation.

\bibitem[HR]{HR} {\sc Hyt\"onen,  T., and Ros\'en, A.}
\newblock  {On the Carleson duality}, 
\newblock {\em Ark. Mat.} (2012), 1--21.

\bibitem[Hu]{Hu} 
{\sc Huang, Y.}
\newblock  Weighted tent spaces with Whitney averages: factorization, interpolation and duality
\newblock  arXiv:1303.5982.

\bibitem[JY]{JY}  
{\sc Jiang, R., and   Yang,  D.}
\newblock New Orlicz-Hardy spaces associated with divergence
form elliptic operators, 
\newblock{\em J. Func. Ana. 258} (2010) 1167--1224.

\bibitem[KM]{KM}
{\sc Kalton, N., and Mitrea, M.}
\newblock Stability results on interpolation scales of quasi-Banach spaces and applications. 
\newblock {\em Trans. Amer. Math. Soc. 350} (1998), no. 10, 3903--3922. 

\bibitem[Ke]{Ke}
{\sc Kenig, C.}
\newblock {\em Harmonic analysis techniques for second order elliptic boundary
  value problems}, vol.~83 of {\em CBMS Regional Conference Series in
  Mathematics}.
\newblock American Mathematical Society, Providence, RI, 1994.

\bibitem[KKPT]{KKPT}
{\sc Kenig, C., Koch, H., Pipher, J., and Toro, T.}
\newblock A new approach to absolute continuity of elliptic measure, with
  applications to non-symmetric equations.
\newblock {\em Adv. Math. 153}, 2 (2000), 231--298.

\bibitem[KP]{KP}
{\sc Kenig, C., and Pipher, J.}
\newblock The {N}eumann problem for elliptic equations with nonsmooth
  coefficients.
\newblock {\em Invent. Math. 113}, 3 (1993), 447--509.

\bibitem[KR]{KR} {\sc Kenig, C., and Rule, D.}
\newblock  {
The regularity and Neumann problem for non-symmetric elliptic operators},
\newblock {\em Trans. Amer. Math. Soc. 361} (2009), no. 1, 125--160.

\bibitem[KS]{KS} {\sc Koskela, P.,  and Saksman, E.}
\newblock Pointwise characterizations of Hardy-Sobolev functions. 
\newblock {\em Math. Res. Lett. 15} (2008), no. 4, 727--744.

\bibitem[LMc]{LMc}   
{\sc Lou, Z., and McIntosh, A.} 
\newblock Hardy spaces of exact forms on $\R^n$. 
\newblock {\em Trans. Am.Math. Soc. 357(4),} (2005),  1469--1496.





%\bibitem{Shen1}
%{\sc Kilty, J., and Shen, Z.}
%\newblock The {$L^p$} regularity problem on {L}ipschitz domains.
%\newblock To appear in: Trans. Amer. Math. Soc.

%\bibitem{Lat}
%{\sc Latter, R.~H.}
%\newblock A characterization of {$H^p(\R^n)$} in terms of atoms.
%\newblock {\em Studia Math. 62}, 1 (1978), 93--101.

%\bibitem{LSW}
%{\sc Littman, W., Stampacchia, G., and Weinberger, H.~F.}
%\newblock Regular points for elliptic equations with discontinuous
%  coefficients.
%\newblock {\em Ann. Scuola Norm. Sup. Pisa (3) 17\/} (1963), 43--77.




\bibitem[MSV]{MSV} 
{\sc Meda, S.,  Sj\"ogren, P., and  Vallarino, M.}
\newblock  On the $H^{1}$ - $L^{1}$ boundedness of operators, 
\newblock {\em Proc. Amer. Math. Soc. 136} (2008), 2921--2931. 

\bibitem[Me]{Me}
{\sc Meyers, N.}
\newblock An {$L^p$}-estimate for the gradient of solutions of second order
  elliptic divergence equations.
\newblock {\em Ann. Scuola Norm. Sup. Pisa. 2}, 17 (1963), 189--206.

\bibitem[Mi]{Mi} 
{\sc  Miyachi, A.} 
\newblock Hardy-Sobolev spaces and maximal functions, 
\newblock {\em J.  Math. Soc. 
 Japan 42 no.1} (1990), 73--90.

\bibitem[Mor]{Mor}
{\sc Morrey, Jr., C.~B.}
\newblock {\em Multiple integrals in the calculus of variations}.
\newblock Die Grundlehren der mathematischen Wissenschaften, Band 130.
  Springer-Verlag New York, Inc., New York, 1966.


\bibitem[Mo]{Mo}
{\sc Moser, J.}
\newblock On {H}arnack's theorem for elliptic differential equations.
\newblock {\em Comm. Pure Appl. Math. 14\/} (1961), 577--591.

\bibitem[N]{N} {\sc Ne\v cas, J.}
\newblock {\it Les m\'ethodes directes en
th\'eorie des \'equations elliptiques. }
\newblock (French) Masson et Cie, Eds.,
Paris; Academia, Editeurs, Prague 1967.

\bibitem[Na]{Na} {\sc Nash, J.}
\newblock  { Continuity of the solutions of parabolic and 
elliptic equations}, {\em Amer. J. Math., 80} (1957), 931-954.

\bibitem[R1]{R1} {\sc Ros\'en, A.}
\newblock {Layer potentials beyond singular integral operators}, 
\newblock arXiv:1210.7582.

\bibitem[R2]{R2} 
{\sc Ros\'en, A.}
\newblock{ Cauchy non-integral formulas}
\newblock {arXiv:1210.7580v1.}

\bibitem[S]{S} {\sc Shen, Z.} 
\newblock A relationship between the Dirichlet and regularity problems for elliptic equations, 
\newblock {\em Math.
Res. Lett. 14 no.2} (2007), 205--213.

%\bibitem{sneiberg}
%{\sc {\v{S}}ne{\u\i}berg, I.~J.}
%\newblock Spectral properties of linear operators in interpolation families of
%  {B}anach spaces.
%\newblock {\em Mat. Issled. 9}, 2(32) (1974), 214--229, 254--255.
%
%\bibitem{SW}
%{\sc Stein, E., and Weiss, G.}
%\newblock On the theory of harmonic functions of several variables. {I}. {T}he
%  theory of {$H\sp{p}$}-spaces.
%\newblock {\em Acta Math. 103\/} (1960), 25--62.
%
%\bibitem{Stein}
%{\sc Stein, E.~M.}
%\newblock {\em Singular integrals and differentiability properties of
%  functions}, vol.~30 of {\em Princeton Mathematical Series}.
%\newblock Princeton University Press, Princeton, NJ, 1970.

\bibitem[Str]{Str}
{\sc Strichartz, R.} 
\newblock $H^p$ Sobolev spaces, 
\newblock {\em Colloq. Math. LX/LXI} (1990), 129--139.

\bibitem[YZ]{YZ} {\sc Yang, D., and  Zhou, Y.}
\newblock{ A Boundedness Criterion via Atoms
for Linear Operators in Hardy Spaces,}  
\newblock {\em Constr. Approx. 29} (2009), no. 2, 207--218.

\end{thebibliography}

\end{document}